\documentclass[reqno,colorlinks=true,citecolor=magenta,linkcolor=blue,11pt]{amsart}

\usepackage{lineno,times}
\modulolinenumbers[5]

\usepackage{hyperref}
\usepackage{amsmath, amsthm, amscd, amsfonts, amssymb, graphicx, color,enumerate}

\newtheorem{main}{Main Theorem}

\numberwithin{equation}{section}
\usepackage{latexsym}
\usepackage{amsmath}
\usepackage{amssymb}

 \usepackage[text={135mm,210mm},centering]{geometry}

\newcommand{\be}{\begin{equation}\begin{aligned}}
\newcommand{\ee}{\end{aligned}\end{equation}}
\newcommand{\ben}{\begin{equation}\nonumber\begin{aligned}}
\newcommand{\dist}{{\rm dist}}
\newcommand{\R}{\mathbb{R}}
\newcommand{\N}{\mathbb{N}}
\newcommand{\D}{\mathcal{D}}
\newcommand{\A}{\mathcal{A}}
\newcommand{\B}{\mathfrak{B}}

\newcommand{\TT}{\mathbb{T}}
\newcommand{\LL}{\mathbb{L}_{per}}

\newcommand{\HH}{\mathbb{H}_{per}}
\renewcommand{\d}{{\rm{d}}}

 \usepackage{thmtools}
\numberwithin{equation}{section}
\newtheorem{theorem}{Theorem}[section]
 \declaretheorem[name=Lemma, sibling=theorem]{lemma}
  
    \declaretheorem[name=Assumption]{assumption}

\theoremstyle{definition}
\newtheorem{definition}[theorem]{Definition}
\theoremstyle{remark}
\newtheorem{remark}[theorem]{Remark}

    \usepackage{mathabx}

\newcommand{\ceqref}[1]{{\color{blue}\eqref{#1}}}
\newcommand{\cref}{\autoref}

\begin{document}

 \title[$(H,H^2)$-smoothing effect of   Navier-Stokes equations]{$(H,H^2)$-smoothing effect  of   Navier-Stokes equations with additive white noise on two-dimensional torus }

\author[H. Cui, H. Liu \& J. Xin]{}

 \keywords{Smoothing effect; $(H,H^2)$-random attractor;  $H^2$ regularity; $H^2$ absorbing set.
 \newline
 \emph{\text{\ \ \quad E-mails. }} h.cui@outlook.com (H. Cui);
 ss\_liuhui@ujn.edu.cn (H. Liu); fdxinjie@sina.com (J. Xin)}

\maketitle

\centerline{\scshape Hongyong Cui$^{1}$, Hui Liu$^2$, Jie Xin$^3$}
\medskip
{\footnotesize
 \centerline{$^1$School of Mathematics and Statistics \& Hubei Key Laboratory of Engineering Modeling and Scientific  Computing,}
 \centerline{Huazhong University of Science and Technology, Wuhan 430074, PR China}}

\medskip
{\footnotesize
 \centerline{$^2$ School of Mathematical Sciences, University of Jinan, Jinan 250022, PR China}

  \medskip
{\footnotesize
 \centerline{$^3$ {School of Information Engineering, Shandong Youth University of Political Science, Jinan 250103, PR China}}

\begin{abstract}

This paper is devoted to the regularity of Navier-Stokes (NS) equations  with  additive white noise  on two-dimensional torus $\mathbb T^2$.  Under the conditions that the external force $f(x)$ belongs to the phase space $ H$ and the noise intensity function $h(x)$  satisfies $\|\nabla h\|_{L^\infty} \leq \sqrt \pi  \nu \lambda_1$,
 where  $ \nu $ is the kinematic viscosity of the fluid and $\lambda_1$ is  the first eigenvalue of the Stokes operator,  it was proved that the random NS equations possess a tempered  $(H,H^2)$-random attractor  whose (box-counting) fractal  dimension in $H^2$ is finite.
This was achieved by establishing, first, an $H^2$ bounded   absorbing set and, second,  an $(H,H^2)$-smoothing effect of the system which lifts the compactness and finite-dimensionality  of the attractor in $H$ to that in $H^2$.   Since the force $f$ belongs only to $H$, the $H^2$-regularity of solutions as well as the $H^2$-bounded absorbing set was constructed by an indirect approach of estimating the $H^2$-distance     between the solution  of the random NS equations and that of the corresponding deterministic equations.
\end{abstract}

\tableofcontents

\section{Introduction}

In this  paper we study the asymptotic dynamics of the following stochastic two-dimensional Navier-Stokes (NS) equations on  $\mathbb{T}^2=[0,L]^2$, $L>0$:
 \begin{align} \label{1}
 \left \{
  \begin{aligned}
 &  \d u_h+ \left (-  \nu \Delta u_h+(u_h\cdot\nabla)u_h+\nabla p\right) \d t=  f(x) \, \d t+h(x)\, \d W,\\
  & \nabla\cdot u_h=0,\\
  & u_h(0)=u_{h,0},
  \end{aligned} \right.
\end{align}
 endowed with periodic boundary conditions, where $u_h$ is the velocity field, $ \nu >0$ is the kinematic viscosity, $p$ is the scalar pressure and $f \in H $  represents volume forces that are applied to the fluid, where $H= \left \{u\in \LL ^2:
  \int_{\mathbb T^2}u \, \d x=0,\,  {\rm div}\, u=0 \right \}$ is the phase space of the system.  $W(t)$ is a two-sided real valued Wiener process on a probability space $(\Omega,\mathcal{F},\mathbb{P})$ with $\Omega=\{\omega\in C(\mathbb{R},\mathbb{R}):\omega(0)=0\}$, $\mathcal{F}$  the Borel $\sigma$-algebra induced by the compact open topology of $\Omega$, and $\mathbb{P}$  the corresponding Wiener measure on $(\Omega,\mathcal{F})$. The subscript $``_{h}"$ indicates the dependence on the noise intensity $h(x)$.
When $h\equiv0$, we obtain the deterministic  NS equations on  $\mathbb{T}^2$:
 \begin{align}\label{2}
 \left \{
  \begin{aligned}&
   \frac{\d u}{\d t}-  \nu \Delta u+(u\cdot\nabla)u+\nabla p=f(x),\\
  & \nabla\cdot u=0,\\
   &u(0)=u_{0}.
  \end{aligned} \right.
\end{align}

The NS equations are fundamental mathematical models in fluid mechanics \cite{temam}. In order to  understand the asymptotic behavior of their dynamics,  many papers have been devoted to the global attractor theory of NS equations \cite{robinson01,temam1}. For the 2D case as we are concerned in this paper,   the global attractor of autonomous NS equations has been studied in \cite{Brzezniak,constantin,ju,liuvx,liuvx1,rosa,sun,wang,yang,zhao}, while \cite{cui24siads,cui24ma,hou,langa,lu,lukaszewicz} and
\cite{Brzezniak13,Brzezniak2015,Brzezniak2018,crauel97jdde,crauel,dong,li,liu2018,mohammed,wang12,wang,wang2023}  studied the global attractors for  non-autonomous and for   stochastic NS equations, respectively, under various settings.   In particular, the  regularity of the global attractor for  \ceqref{2} was studied in detail in \cite{robinson01}, and it is known that if $f\in H^s$, where    $H^s $, $s>0$, are standard functional spaces in the theory of NS equations (see   \cref{sec3} for the setting),  then the global attractor $\A_0$ of  deterministic NS equations is bounded in $H^{s+1}$,
see, e.g., \cite{robinson13}.   Particularly for $f\in H$,
 the  following  $H^2$-boundedness of the global attractor  $\A_0$  of   \ceqref{2}  is well-known, see also \cite{julia} for a non-autonomous version of pullback attractors.
\begin{lemma}(\cite[Proposition 12.4]{robinson01})  \label{intro-rob}
For $f\in H$ the deterministic NS equations  \ceqref{2}  generates a semigroup $S$ which possesses a global attractor $\A_0$ in $H$ and has an
  absorbing set which is bounded in $H^2$. In particular,    the global attractor $\A_0$   is   bounded in $H^2$.
\end{lemma}

 This means the global attractor $\A_0$ of  NS equations  \ceqref{2} is a global $(H,V)$-attractor, where $V=H^1$, that is,  it is a  compact set in $V$
 and  attracts  every bounded set in $H$ under the norm of $V$, indicating that the global attractor $\A_0$ consists of strong solutions of the equation, see \cite[Theorem 9.5]{robinson01}, and every weak solution of the equation is attracted by these strong solutions in $V$.   Further, the $H^2$-boundedness of the attractor $\A_0$ indicates  that the strong solutions in $\A_0$ are in fact smooth solutions  (\cite{Ladyzenskaja,robinson01,temam}), but   the attraction of the attractor happens only in $V$ since so far the attractor  has not been proved to be $(H,H^2)$.

    In fact,  the  $(H,V)$-attractor of 2D NS equations can exist in more general  settings.     For instance,  employing the technique of asymptotic compactness instead of the compactness embeddings,   \cite{ju}  established the $(H,V)$-global attractor for  some unbounded domains.  Adapted  techniques of asymptotic compactness have been applied to non-autonomous and random cases. For instance,  \cite{cui21jde}  studied the construction and finite fractal dimensionality of $(H,V)$-uniform attractors and    \cite{zhao19} established $(H,V)$-trajectory attractors  for non-autonomous NS equations, while \cite{li} studied  the $(H,V)$-random uniform attractor for non-autonomous and random NS equations.  These techniques for obtaining $(H,V)$-type attractors have been successfully applied to other fluid mechanics  models as well, see, e.g.,  most recently  \cite{wang2024,wang2023}.

It is possible to prove that  the global attractor $\A_0$ of deterministic NS equations  \ceqref{2} is in fact an $(H,H^2)$-attractor
 by  a  technique  of, e.g.,  deriving the $H^1$ boundedness of the time-derivatives $u_t$ of solutions as  Song \& Hou \cite{song}  did for 3D damped NS equations.  Nevertheless,  in  this paper we shall work on the more challenging  stochastic case  \ceqref{1} instead of a deterministic case, proving that under certain conditions  the random dynamical system generated by  \ceqref{1} has  an $(H,H^2)$-random attractor, and, in addition, this random attractor  has finite (box-counting) fractal dimension in $H^2$.  Since the  idea of estimating $u_t$  does not apply to stochastic equations, even the $H^2$-random absorbing sets need be constructed differently here from \cite{song}.

       \subsection*{Main results. Main techniques}

  The main results of the paper   are now  introduced.
To begin with, following Crauel  \& Flandoli \cite{crauel}  we reformulated the stochastic NS   equation  \ceqref{1}   via a scalar Ornstein-Uhlenbeck process as the following random equation
\begin{align}  \label{1.3}
\frac{\d v}{\d t}+\nu Av+B\big( v(t)+hz(\theta_t\omega)\big)=f(x)-\nu Ahz(\theta_t\omega)+hz(\theta_t\omega),
\end{align}
where $A =-P\Delta$, $B(u)=P((u\cdot \nabla)u)$ with $P:\LL^2\to H$ the Helmholtz-Leray orthogonal projection, and   $z$ is a tempered random variable.  This random equation is known as a conjugate system to  \ceqref{1} and generates a random dynamical system (RDS) $\phi$ in $H$, see  \cref{sec2.1} for more detail.
The intensity  function of the noise $h(x)$ is supposed to satisfy  the following assumption.

\begin{assumption} \label{assum}
Let    $ h \in H^3$, and
\[
  \|\nabla h\|_{L^\infty} <  {\sqrt \pi} \nu \lambda_1 ,
\]
where  $\lambda_1 = 4\pi^2 /L^2$ is the first eigenvalue of the Stokes operator $A$.
  \end{assumption}

  In   \cref{theorem4.1} we  constructed an $H^2$-random absorbing set as the following.

\begin{main}[$H^2$ random absorbing set]
Let $f\in H$ and  \cref{assum} hold. Then the RDS $\phi$  has a  random absorbing set $\mathfrak{B}_{H^2}$ given as a random $H^2$ neighborhood of the global attractor $\mathcal{A}_0$ of the deterministic equation  \ceqref{2}:
\begin{align}
\mathfrak{B}_{H^2}(\omega)=\left \{v\in H^2: \, {\rm dist}_{H^2}(v,\mathcal{A}_{0}) \leq \sqrt{\rho(\omega)}  \, \right\}, \quad \omega\in\Omega, \nonumber
\end{align}
where $\rho(\cdot)$ is a tempered random variable.
 As a consequence, the random attractor $\mathcal{A}$ of $\phi $ is a bounded and tempered random set in $H^2$.
\end{main}
This result  generalizes \cref{intro-rob} to random cases, while  a major technical difficulty arises  here   that   the standard techniques used for the deterministic NS equations in obtaining the $H^2$ regularity   do not apply to the random equation  \ceqref{1.3}.  More precisely,   \cref{intro-rob} was achieved by  estimating the  time-derivative $u_t$ of solutions, see \cite{robinson01}. This, however, cannot be done for  random equation \ceqref{1.3} since the Winner process is not derivable in time.  On the other hand,  a more  direct method of estimating the $H^2$-norm of $v$  by multiplying both sides of  equation  \ceqref{1.3} by $A^2 v$ is not applicable either since $f$ belongs only to $H$.

In order to overcome this difficulty we employed a comparison approach, adapted from an idea of  Sun \cite{sun}  used in a study of the regularity of  deterministic wave equations. We
 estimated  the  difference $w:=v-u$  between the solutions of the random and the deterministic   NS equations instead of estimating the solution $v$ itself, where $u$ is taken as a solution corresponding to an initial value lying  in the global attractor $\A_0$.
 By proving  that $w$ is asymptotically bounded in $H^2$ (see   \cref{lemma4.4})  it follows the desired $H^2$ random absorbing set of $v$ since  $ \cup_{t\geq 0} u(t) \subset \A_0$ has been  bounded  in $ H^2$ (see   \cref{intro-rob}).
   \medskip

After the $H^2$ regularity ($H^2$ boundedness)  of the random attractor $\A$ has been established  we were interested in its compactness in $H^2$. In   \cref{theorem5.1}  we proved that   the random attractor $\A$
  is not only compact but also finite-dimensional in $H^2$, and the attraction  happens in $H^2$ as well.  More precisely,

\begin{main}[$(H,H^2)$-random attractor] Let $f\in H$ and   \cref{assum} hold. Then the random attractor $\mathcal{A}$ of the random NS equation  is an    $(H,H^2)$-random attractor  of finite fractal dimension in $H^2$.
\end{main}

This theorem follows from standard bi-spatial random attractor theory \cite{cui18jdde}, where a key    step in proving an attractor to be  $(H,H^2)$  is to prove the  $(H,H^2)$-asymptotic compactness of the system.   Nevertheless,  instead of the required $(H,H^2)$-asymptotic compactness of the system, we here proved a more technical   $(H,H^2)$-smoothing  property. This smoothing property is essentially a local $(H,H^2)$-Lipschitz continuity in initial values, and  has been shown powerful in estimating the fractal dimension of attractors and constructing exponential attractors, see for instance Cholewa et al$.$ \cite{cholewa08}, most recently  Carvalho et al$.$ \cite{carvalho25jns} and references therein.  In addition, once the  finite-dimensionality of the attractor in $H$ has been known (as in our case), from  the $(H, H^2)$-smoothing property  it would follow immediately the
finite-dimensionality in $H^2$.
Therefore,  in   \cref{theorem5.6} we derived an $(H,H^2)$-smoothing property as stated below.

\begin{main}[$(H,H^2)$-smoothing effect] Let $f\in H$ and  \cref{assum} hold.
 Then for   any tempered set $\mathfrak D $ in $H$ there
  exist  random variables $T_{{\mathfrak D}} (\cdot)  $   and $ L_{\mathfrak D}(\cdot )$ such that  any two solutions $v_1$ and $v_2$ of random NS equations  \ceqref{1.3} corresponding to initial values   $v_{1,0},$ $ v_{2,0}$ in $\mathfrak D \left (\theta_{-T_{\mathfrak D}(\omega)}\omega \right)$, respectively, satisfy
  \ben
  &
  \left\|v_1 \!  \left (T_{\mathfrak D}(\omega),\theta_{-T_{\mathfrak D}(\omega)}\omega,v_{1,0}\right)
  -v_2 \! \left (T_{\mathfrak D}(\omega),\theta_{-T_{\mathfrak D}(\omega)}\omega,v_{2,0} \right) \right \|^2_{H^2} \\[0.8ex]
&\quad \leq  L_{\mathfrak D} (\omega)\|v_{1,0}-v_{2,0}\|^2_H ,\quad \omega\in \Omega.
\ee
\end{main}

Note that the proof of this theorem  is  technical and the periodic boundary condition greatly facilitated the analysis.  We believe that the main idea of this paper applies to smooth bounded domains or some unbounded domains  under   proper boundary conditions, while more technical calculations and analysis would be expected.
Since multiplicative noises would lead to a special  structure of  the noise being coupled, our idea of the comparison approach seems not applicable for multiplicative noises. Hence, the $H^2$ random absorbing sets as well as the $H^2$ random attractor of   NS equations with   multiplicative noises  remain open.

\subsection*{Organization of the paper}

 In   \cref{sec2} we recall the required $(X,Y)$-random attractor theory and their fractal dimensions.  In   \cref{sec3}, we first  introduce the  functional spaces that are necessary to study NS equations, and then associate an RDS to the random equation.   Though the  random attractor in  phase space $H$ has been well-known, it is  carefully derived in this section since we are under  new   conditions and,  in addition,  the estimates obtained here will be crucial  afterwards in constructing  $H^2$ random absorbing sets and the $(H,H^2)$-smoothing effect of  the  RDS.  In  \cref{sec4}, by an indirect approach of  estimating the $H^2$-distance between the random and the deterministic solution trajectories  we construct an  $H^2$ random absorbing set. In   \cref{sec5}, we  prove the crucial  $(H,H^2)$-smoothing  of the RDS from which it follows in   \cref{sec6}  that the random attractor    in $H$ is in fact an $(H,H^2)$-random attractor of finite  fractal dimension in $H^2$.

\section{Preliminaries:   $(X,Y)$-random attractors and  their fractal dimension}\label{sec2}

In this section we introduce the basic  theory of $(X,Y)$-random attractors and their fractal dimension needed in this paper.   The readers are referred to \cite{arnold,crauel,cui18jdde,langa2} and references therein.

Let $(X,\|\cdot\|_X)$ be a Banach space  and  $(\Omega,\mathcal{F},\mathbb{P})$ be  a probability space with a group $\{\theta_t\}_{t\in\mathbb{R}}$ of measure-preserving self-transformations on $\Omega$.
Denote by $\mathcal{B}(\cdot)$ the Borel $\sigma$-algebra of a metric space, and let
   $\R^+:=[0,\infty)$.

\begin{definition}
A mapping $\phi:\mathbb{R}^+\times\Omega\times X\rightarrow X$ is called a {\em random dynamical system (RDS)}   in $X$, if
\begin{enumerate}[(i)]
\item $\phi$ is $(\mathcal{B}(\mathbb{R}^+)\times\mathcal{F}\times\mathcal{B}(X),\mathcal{B}(X))$-measurable;

\item  for every $\omega\in\Omega$, $\phi(0,\omega,\cdot)$ is the identity on $X$;

\item $\phi$ satisfies  the following cocycle property
\begin{align*}
\phi(t+s,\omega,x)=\phi(t,\theta_s\omega,\phi(s,\omega,x)), \quad \forall t,s\in \mathbb{R}^+,~\omega\in\Omega,~x\in X;
\end{align*}
\item the mapping $x\mapsto \phi(t,\omega, x)$ is continuous.
\end{enumerate}
\end{definition}

A  set-valued mapping $\mathfrak D$: $\Omega\mapsto 2^X\setminus \emptyset $, $ \omega\mapsto \mathfrak D(\omega) $, is called a \emph{random set}  in $X$ if it is   \emph{measurable} in the sense that the mapping $\omega\to \dist_X(x, \mathfrak D(\omega))$    is  $(\mathcal F,\mathcal B(\R))$-measurable for each $x\in X$. If each its image $\mathfrak D(\omega)$ is closed (or  bounded, compact, etc.) in $X$, then $\mathfrak D$ is called a  closed (or  bounded, compact, etc.) random set in $X$.

A random variable $\zeta:\Omega\mapsto X$ is called {\em tempered} if
\[
  \lim_{t\to \infty} e^{-\varepsilon t} |\zeta(\theta_{-t}\omega)| =0, \quad \varepsilon >0.
\]
Note that if $\zeta $ is a  tempered random variable, then so is $|\zeta|^k$ for all $k\in \mathbb N$.
A random set $\mathfrak D$ in $X$  is called {\emph{tempered}}  if there exists a   tempered random variable $\zeta$ such that $\|\mathfrak D(\omega)\|_X:=\sup_{x\in \mathfrak D(\omega)} \|x\| \leq \zeta(\omega)$ for all $\omega\in \Omega$.

\smallskip
Denote by $\mathcal{D}_X$  the collection of all the tempered random sets in $X$.

 \begin{definition}
A    random set $\mathfrak{B}$ in $X$ is called a {\emph{(random)  absorbing set}} of an RDS $\phi$ if  it pullback absorbs every tempered random set in $X$, i.e., if for any  $\mathfrak D\in\mathcal{D}_X$   there exists a random variable $T_{\mathfrak D} (\cdot)$ such that
\begin{align*}
\phi(t,\theta_{-t}\omega, \mathfrak D(\theta_{-t}\omega))\subset \mathfrak{B}(\omega),\quad t\geq T_{\mathfrak D}(\omega),\  \omega\in\Omega.
\end{align*}
\end{definition}

\begin{definition}
A random set $\mathcal{A}$  is called  the \emph{random attractor}   of an RDS  $\phi$ in $X$ if
\begin{enumerate}[(i)]
\item  $\mathcal{A}$ is a tempered and compact random set in $X$;

 \item  $\mathcal{A}$ is invariant under $\phi$, i.e.,
\begin{align*}
\mathcal{A}(\theta_t\omega)=\phi(t,\omega,\mathcal{A}(\omega)), \quad t\geq0,~\omega\in\Omega;
\end{align*}

 \item $\mathcal{A}$  pullback attracts every tempered set in $X$, i.e.,  for any $\mathfrak D\in  \mathcal{D}_X$,
 \[
\lim\limits_{t\rightarrow\infty} {\rm dist}_{X} \big(\phi(t,\theta_{-t}\omega,\mathfrak D(\theta_{-t}\omega)), \, \mathcal{A}(\omega) \big)=0,\quad \omega\in\Omega ,
 \]
where
$\dist _{X}(\cdot, \cdot) $  denotes  the Hausdorff semi-metric between   subsets  of  $ X$, i.e.,
\[
 \dist_{X}(A,B)= \sup_{a\in A} \inf_{b\in B} \|a-b\|_X, \quad \forall A,B\subset X.
\]
\end{enumerate}
\end{definition}

Let $Y$ be a Banach space with continuous embedding $Y \hookrightarrow X$.
\begin{definition}
A random set $\mathcal{A}$   is called  the {\emph{$(X,Y)$-random attractor}}  of  an RDS $\phi$ in $X$,  if
\begin{enumerate}[(i)]
\item  $\mathcal{A}$ is a tempered and compact random set in $Y$;

 \item  $\mathcal{A}$ is invariant under $\phi$, i.e.,
\begin{align*}
\mathcal{A}(\theta_t\omega)=\phi(t,\omega,\mathcal{A}(\omega)), \quad t\geq0,~\omega\in\Omega;
\end{align*}

 \item $\mathcal{A}$  pullback attracts every tempered set in $X$ in the metric of $Y$, i.e.,  for any $\mathfrak D\in  \mathcal{D}_X$,
 \[
\lim\limits_{t\rightarrow\infty} {\rm dist}_{Y} \big(\phi(t,\theta_{-t}\omega, \mathfrak D(\theta_{-t}\omega)), \, \mathcal{A}(\omega) \big)=0,\quad \omega\in\Omega .
 \]
\end{enumerate}
\end{definition}

Note that  in addition to the temperedness, compactness and pullback attraction in $Y$,  the measurability in $Y$ of an $(X,Y)$-random attractor was also  required by definition.

  The following criterion  is an adaption of \cite[Theorem 19]{cui18jdde}.
  \begin{lemma}(\cite[Theorem 19]{cui18jdde}) \label{lem:cui18}
  Let $\phi$ be an RDS.  Suppose that
  \begin{enumerate}[(i)]
  \item  $\phi$ has a   random absorbing set $\mathfrak B$ which is a tempered and closed random set in $Y$;
  \item $\phi$  is $(X,Y)$-asymptotically compact, that is, for any $ \mathfrak D\in \D_X$,  $x_n \in \mathfrak D(\theta_{-t_n} \omega)$
  and $t_n\to \infty$,   the sequence
  $\{\phi(t_n, \theta_{-t_n} \omega, x_n)\}_{n\in \N}$  has a convergent subsequence in $Y$.
  \end{enumerate}
    Then  $\phi$ has a unique   $(X,Y)$-random attractor $\A$ given by
    \[
    \A(\omega) = \bigcap_{s\geq 0} \overline{ \bigcup_{t\geq  s} \phi(t,\theta_{-t} \omega, \mathfrak  B(\theta_{-t}\omega)) }^Y ,
    \quad \omega \in \Omega.
    \]
  \end{lemma}

 \medskip

 Suppose that $E$ is a precompact  set in $X$, and let $N(E, \varepsilon)$ denote the minimum number of balls of radius $\varepsilon$ in $X$ required to cover $E$. Then the {\emph{(box-counting) fractal dimension of $E$}}  is defined as
\[
 d_f^X(E)=\limsup_{\varepsilon\to 0} \frac{ \log N(E,\varepsilon)}{ -\log \varepsilon} ,
\]
where the superscript ``$\ ^X$''   indicates the space to which  the fractal dimension is referred.

The following lemma is a random adaption of a well-known property of fractal dimension under Hölder continuous projections, see, e.g., Robinson \cite{robinson11}.
\begin{lemma}(\cite[Lemma 5]{cui}) \label{lem:cui}
Let  $\phi$ be an RDS in $X$ with a random attractor $\A$ of finite fractal dimension. If $\phi$ has a tempered  random absorbing set $\B $ which is bounded in $Y$ and there are   random variables $T_\omega$ and $L(\omega)$ such that  the following $(X,Y)$-smoothing property on $\B$
 \ben
 \|\phi(T_\omega,\theta_{-T_\omega} \omega, x_1)-\phi(T_\omega,\theta_{-T_\omega}\omega, x_2)\|_Y \leq L(\omega) \|x_1-x_2\|_X^\delta,\quad x_1,x_2\in \B(\theta_{-T_\omega} \omega),
 \ee
 holds for some $\delta>0$ for all $\omega\in \Omega$,
 then  $\A$ has finite fractal dimension in $Y$:
 \[
 d_f^Y \big(\A(\omega) \big)\leq \frac1\delta \,  d_f^X\big(\A(\theta_{-T_\omega}\omega)\big) ,\quad \omega\in \Omega.
 \]
 \end{lemma}

Notet that in the case of $\{\theta_t\}_{t\in \R}$ being  ergodic  and  the RDS $\phi$ being Lipschitz in initial values, the fractal dimension  of the random attractor is $\mathbb P$-a.s$.$ constant, see Langa \& Robinson \cite{langa2}.

\section{The Navier-Stokes equations and the   attractors in $H$}\label{sec3}

\subsection{Functional spaces}

 In order to study the periodic boundary conditions
  we  first
 introduce Sobolev spaces of periodic functions.  Such functions can be represented as Fourier series, which makes their analysis significantly more straightforward than for the spaces on bounded domains.  The following settings are standard, the readers are referred to, e.g., \cite{robinson01,temam1,temam}. etc.

  Denote by $(C_{per}^\infty(\mathbb T^2))^2 $  the space of two-component infinitely differentiable functions that are $L$-periodic in each direction, and by  $\LL^2$, $\HH^1$ and $\HH^2$ the completion of $(C_{per}^\infty(\mathbb T^2))^2 $  with respect to the $(L^2(\TT^2))^2$, $(H^1(\TT^2))^2$ and $(H^2(\TT^2))^2$ norm, respectively.

 The phase space $(H,\|\cdot\|)$ is a subspace of  $\LL^2$ of functions with zero mean and free divergence, i.e.,
\begin{align*}
H= \left \{u\in \LL ^2:
\, \int_{\mathbb T^2}u \ \d x=0,\  {\rm div}\, u=0 \right \},
\end{align*}
endowed with the norm of $\LL^2$, i.e.,
  $\|\cdot\|=\|\cdot\|_{\LL^2}$.   For $p>2$,  we write $(L^p(\TT^2))^2$  simply as $L^p$ and its norm as $\|\cdot\|_{L^p}$.

Note that any $u\in \LL^2 $ has an expression $u=\sum_{j\in \mathbb{Z}^2} \hat u_j e^{{\rm i}j\cdot x}$, where  $ {\rm i}= \sqrt{-1}$ and $\hat u_j$ are Fourier coefficients.  Thus, $\int_{\mathbb T^2}    u \,  \d x=0$ is equivalent to $\hat u_0=0$. Therefore, any  $u\in H $ is in form
\[
 u =\sum_{j\in \mathbb Z^2\setminus\{0\}} \hat u_j e^{{\rm i}j\cdot x} .
\]
 Let $P: \LL ^2\rightarrow H$ be the Helmholtz-Leray orthogonal projection operator. In  this periodic space, we define the stokes operator  $Au=-P\Delta u=-\Delta u$ for all $u\in D(A) $. In  this bounded domain, the projector $P$ does not commute with derivatives. Then the operator $A^{-1}$  is a self-adjoint positive-definite compact operator from $H$  to $H$.

   For $s >0 $,    the Sobolev space  $H^s:=D(A^{s/2})$  is defined in a standard way by
 \[
H^s =\left \{u\in H:\|u\|^2_{H^s}=\sum\limits_{j\in \mathbb{Z}^{2}\backslash \{0\}}|j|^{2s}|\hat{u}_j|^2<\infty \right \},
 \]
endowed with the norm $\|\cdot\|_{H^s}=\|A^{s/2}\cdot \|$.

For $u,v\in H^1$,  we define the bilinear form
\[
B(u,v)=P((u\cdot\nabla)v) ,
\]
and, in particular, $B(u):=B(u,u)$.

 The following Poincar\'e's inequality is standard, see, e.g., \cite[lemma 5.40]{robinson01}.

\begin{lemma}[Poincar\'e's inequality]  If $u\in H^1$, then
\begin{equation} \label{poin}
 \|u\| \leq \left( \frac L{2\pi} \right) \|    u\|_{H^1}.
\end{equation}
\end{lemma}
Thus, with
\[
 \lambda_1 :=   \frac{4\pi^2 } {L^2}  ,
\]
we have $\| u\|_{H^1}^2 \geq\lambda_1 \|u\|^2$. For later purpose we  denote by $\lambda:= \alpha \nu \lambda_1/4$.

We let  the noise  intensity $h$  satisfy   \cref{assum}, i.e.,     $ h \in H^3$ and
\begin{align}\label{2.3}
  \|\nabla h\|_{L^\infty} <  {\sqrt \pi} \nu \lambda_1 .
\end{align}
 (In particular, $\|\nabla h\|_{L^\infty} < \sqrt{\pi} \nu $  for the scale $L=2\pi$).
Clearly, the  tolerance of the noise  intensity     increases as the kinematic viscosity $\nu$ increases.

\subsection{Global attractor of the deterministic Navier-Stokes equation}

We   now  recall the global attractor of the deterministic NS equations  \ceqref{2}, i.e. for $h(x)\equiv 0$,  which will be crucial for our analysis later. Under projection $P$, the      equation   is rewritten as
 \begin{align}\label{2.1}
   \partial_tu+ \nu Au+B(u)  =f,\quad u(0)=u_{0}.
\end{align}
This equation   generates an autonomous dynamical system $S$ in phase space $H$ with a global attractor $\A_0$. In addition,    the following regularity result of the attractor $\A_0$ seems  optimal for $f\in H$  in the literature.

\begin{lemma}\label{lem:det} (\cite[Proposition 12.4]{robinson01})
For $f\in H$ the deterministic NS equations  \ceqref{2.1}  generates a semigroup $S$ which possesses a global attractor $\A_0$ in $H$ and has an
  absorbing set which is bounded in $H^2$. In particular,    the global attractor $\A_0$   is   bounded in $H^2$.
\end{lemma}

We will prove that the global attractor $\mathcal A$ is not only bounded in $H^2$, but also has   finite fractal dimension (and thus compact) in $H^2$.

\begin{remark}  \label{rmk1}
 As $f$ belongs only to $H$, one cannot obtain an energy equation of  $\|Au\|$ directly and thus  \cref{lem:det} was proved by estimating the time-derivatives $ \partial_t u $  of solutions    rather than estimating the solutions  $u$ themselves.       However, this   method  does not apply to the  random NS  equation  \ceqref{2.2}  since the Wiener process is not derivable in time.  Hence, in  \cref{sec4} we will employ a comparison approach  to obtain $H^2$ random absorbing sets for  \ceqref{2.2}.
\end{remark}

\subsection{Generation of an RDS}\label{sec2.1}

We now transform the stochastic Navier-Stokes equation  \ceqref{1}   to a random NS equation  which    generates an RDS $\phi$ on $H$.
          Following Crauel, Defussche \& Flandoli \cite{crauel97jdde},    let
 \[
z(\theta_t\omega)=-\int_{-\infty}^0e^\tau(\theta_t\omega)(\tau)\, \d \tau,\quad \forall t\in \mathbb{R},\ \omega\in\Omega.
\]
Then $z(\theta_t\omega)$ is   the one-dimensional Ornstein-Uhlenbeck process which solves the equation
 \[
\d z(\theta_t\omega)+z(\theta_t\omega) \, \d t= \d \omega.
 \]
Moreover, there exists a $\theta_t $-invariant  subset $\tilde \Omega\subset \Omega$ of full measure such that $z(\theta_t\omega)$ is continuous in $t$  for every $\omega \in \tilde\Omega$ and the random variable $|z(\cdot )|$ is tempered, namely, for each $\varepsilon>0$ it holds
\[
   \lim_{t \to \infty} e^{-\varepsilon t} | z(\theta_{-t}\omega )| =0,\quad \forall  \omega\in \Omega .
\]
In fact, it satisfies
\[
\lim_{t\to \pm\infty}\frac{|z(\theta_t\omega)|}{|t|}=0,\quad \lim_{t\to \pm\infty}\frac 1t \int^t_0 z(\theta_s\omega)\, \d s =0.
\]
Moreover, by the ergodicity of the  parametric system, see, e.g., \cite{arnold} or \cite{wang2024}, we have
\begin{align} \label{erg}
\lim\limits_{t\rightarrow \pm \infty}\frac{1}{t}\int_0^t|z(\theta_s\omega)|^m\, \d s=\mathbb{E}  \big (|z(\theta_t\omega)|^m
\big) =\frac{\Gamma  (\frac{1+m}{2} )}{\sqrt{\pi}}
,\quad \omega\in \tilde \Omega,
\end{align}
for  all $m\geq 1$, where $\Gamma$ is the Gamma function.  Hereafter, we will not distinguish $\tilde \Omega$ and $\Omega$.

 We introduce the following transformation
 \[
u_h(t)=v(t)+hz(\theta_t\omega),\quad t>0,\ \omega\in \Omega.
 \]
Under projection $P$, the  system   \ceqref{1} of $u_h$  is then transformed to   the   abstract evolution equations
\begin{align}\label{2.2}
\left\{
\begin{aligned}
&
\frac{\d v}{\d t}+ \nu Av+B\big( v(t)+hz(\theta_t\omega)\big)=f(x)-  \nu Ahz(\theta_t\omega)+hz(\theta_t\omega), \\
& v(0)=u_h(0)-hz(\omega),
\end{aligned}
\right.
\end{align}
 of $v$ in $H$.

By virtue of the standard Galerkin method, we have the   well-posedness  of weak solutions of  \ceqref{2.2} as stated below.
\begin{lemma}
Let $f\in H$ and  \cref{assum} hold. Then for each $v(0)\in H$ and $\omega\in\Omega$, there exists a unique weak solution
\begin{align*}
v (t) \in C_{loc}([0,\infty);H)\cap L^2_{loc} (0,\infty;H^1 )
\end{align*}
satisfying  \ceqref{2.2} in distribution sense with $v|_{t=0}=v(0)$.  In addition, this solution is continuous in initial values in $H$.
\end{lemma}

Therefore, by
\[
 \phi(t,\omega, v_0) =v(t,\omega,v_0),\quad  \forall t\geq 0,\, \omega\in \Omega, \, v_0\in H,
\]
where $v$ is the solution of  \ceqref{2.2},
we associated an RDS $\phi$ in $H$ to  \ceqref{2.2}.

\subsection{The random attractor in $H$}
We now construct the random attractor $\A$ in $H$ for the RDS $\phi$.  Though the construction here is essentially standard,   the detail will be given   since we are under a new   condition  \ceqref{2.3} in  \cref{assum} and,  in addition,  the bounds obtained here will be crucial later in constructing  $H^2$ random absorbing sets and the local $(H,H^2)$-Lipschitz continuity of  $\phi$.

 We first make some uniform estimates on solutions  of  \ceqref{2.2} corresponding to initial values $v(0)$ in $ H$,    with    careful use of the condition  \ceqref{2.3} in  \cref{assum}.   Throughout  this paper, we let $C$  denote a positive constant whose value may vary  from line to line.

  We will frequently make use of the following two    fundamental lemmas,  see, e.g., \cite{temam1,temam}.

\begin{lemma}[{Gronwall's lemma}]
Let $x(t)$ be  a function from $\R$ to $\R^+$ such that
\[
\dot x +a(t)x\leq b(t) .
\]
 Then for all $t\geq s $,
\[
 x(t)\leq e^{ -\int^t_s a(\tau) \, \d \tau} x(s)+\int_s^t e^{-\int^t_\eta a(\tau) \, \d \tau } b(\eta) \, \d \eta.
\]
\end{lemma}

\begin{lemma}[Gagliardo-Nirenberg inequality]  If $u\in L^q(\mathbb{T}^2)$, $D^{m}u\in L^r(\mathbb{T}^2)$, $1\leq q,r\leq\infty$, then there exists a constant $C$ such that
\[
 \|D^ju\|_{L^p} \leq  C \|D^mu\|_{L^r}^a\|u\|^{1-a}_{L^q},
\]
where
\begin{equation*}
\frac{1}{p}-\frac{j}{2}=a \left (\frac{1}{r}-\frac{m}{2} \right)+ \frac{1-a}{q}, \quad 1\leq p\leq\infty,~0\leq j\leq m,~\frac{j}{m}\leq a\leq1,
\end{equation*}  and
$C$ depends only on $\{m, \, j, \, a, \, q, \,r \}$.
\end{lemma}

\begin{lemma}[$H$ bound]\label{lemma4.1}
Let  $f\in H$ and  \cref{assum} hold.
Then there are random variables  $T(\omega)$ and $\zeta_1(\omega)$,   where $\zeta_1$ is tempered, such that any
  solution $v$  of the system  \ceqref{2.2} corresponding to initial  values  $v(0) \in H $ satisfies
  \begin{align}
 \|v(t,\theta_{-t} \omega, v(0))\|^2
 \leq e^{-\lambda t}\|v(0)\|^2
 + \zeta_1(\omega) ,\quad t\geq T_1(\omega).  \nonumber
\end{align}
\end{lemma}
\begin{proof}
Note that since,  by  \ceqref{2.3}, $  {\|\nabla h\|_{L^\infty} }/{\sqrt \pi} < \nu  \lambda_1$, there exist  $  \alpha \in  (0,1]$ and $\beta>0$  such that
\begin{gather}
  \frac{\|\nabla h\|_{L^\infty} }{\sqrt \pi}  =  (1- \alpha) \nu  \lambda_1,  \label{c1} \\
   \frac{  \|\nabla h\|_{L^\infty}  }{ \sqrt{\pi}}  \left( 1 +  \beta \right )  =   \nu \lambda_1 -\frac  1 2 \alpha \nu  \lambda_1  .  \label{c2}
  \end{gather}
 Then multiplying the  equation    \ceqref{2.2} by $v$ and integrating over $\mathbb T^2$, by integration by parts we have
\[
\frac12\frac{\d}{\d t}\|v\|^2+ \nu  \|A^{\frac12}v\|^2
 =-(B(v+hz(\theta_t\omega)) ,v)  + \big (f- \nu Ahz(\theta_t\omega) +hz(\theta_t\omega),v\big),
\]
so by Poincar\'e's inequality  \ceqref{poin},
\begin{align}
 & \frac{\d}{\d t}\|v\|^2+ \left(2-\frac \alpha 2\right)  \nu  \lambda_1  \| v\|^2  + \frac {\alpha \nu } 2   \|A^{\frac12}v\|^2 \nonumber  \\
 &\quad \leq \frac{\d}{\d t}\|v\|^2  + 2\nu  \|A^{\frac12}v\|^2 \nonumber\\
&\quad  \leq 2 \big |\big(B(v+hz(\theta_t\omega)) ,v \big) \big|
+2\big| \big (f- \nu Ahz(\theta_t\omega) +hz(\theta_t\omega),v \big)\big|  . \label{4.3}
\end{align}
Applying H\"{o}lder's, Young's and Poincar\'{e}'s inequalities, we obtain
\begin{align}
   2\big | \big(B(v+hz(\theta_t\omega)) ,v \big) \big|   & \leq  2 \big | \big(B(v,hz(\theta_t\omega)),v \big) \big| +2 \big| \big (B(hz(\theta_t\omega),hz(\theta_t\omega)),v \big) \big| \nonumber\\
&\leq 2 |z(\theta_t\omega)|\int_{\mathbb{T}^2}|v|^2|\nabla h|\ \d x+ 2|z(\theta_t\omega)|^2\int_{\mathbb{T}^2}|v||h||\nabla h|\ \d x\nonumber\\
&\leq 2 |z(\theta_t\omega)|\|\nabla h\|_{L^\infty}\|v\|^2+ 2|z(\theta_t\omega)|^2\|\nabla h\|_{L^\infty}\|v\|\|h\|\nonumber\\
&\leq 2 |z(\theta_t\omega)|\|\nabla h\|_{L^\infty}\|v\|^2+  \frac{\alpha \nu \lambda_1} {8} \|v\|^2+ C|z(\theta_t\omega)|^4\|\nabla h\|_{L^\infty}^2\|h\|^2 , \nonumber
\end{align}
and
\begin{align}
 2\big| \big (f-\nu Ahz(\theta_t\omega) +hz(\theta_t\omega),v \big)\big|  \leq \frac{\alpha \nu \lambda_1}{8}\|v\|^2+C\|f\|^2
 + C|z(\theta_t\omega)|^2 \left(\|Ah\|^2 + \|h\|^2\right) .\nonumber
\end{align}
Substituting these inequalities   into  \ceqref{4.3} yields
\ben
&
\frac{\d}{\d t}\|v\|^2+ \left(2-\frac {3 \alpha} 4\right) \nu \lambda_1  \| v\|^2  + \frac {\alpha \nu }  2 \|A^{\frac12}v\|^2  \\
 &\quad \leq 2|z(\theta_t\omega)|\|\nabla h\|_{L^\infty}\|v\|^2+ C |z(\theta_t\omega)|^4\|\nabla h\|_{L^\infty}^2\|h\|^2\nonumber\\
&\qquad + C\|f\|^2+ C |z(\theta_t\omega)|^2 \left (\|h\|^2+\|Ah\|^2 \right)\nonumber\\
&\quad \leq 2|z(\theta_t\omega)|\|\nabla h\|_{L^\infty}\|v\|^2+C \left (1+|z(\theta_t\omega)|^4 \right).
\ee
Therefore,
\begin{align*}
\frac{\d}{\d t}\|v\|^2+ \left[ \left(2-\frac{3\alpha} 4\right ) \nu \lambda_1 -2\|\nabla h\|_{L^\infty}|z(\theta_t\omega)| \right ]
\|v\|^2+\frac{ \alpha \nu } 2\|A^{\frac12}v\|^2
\leq C  \left (1+|z(\theta_t\omega)|^4 \right)  .
\end{align*}
Applying  Gronwall's lemma, we obtain
\begin{align}
&\|v(t)\|^2+\frac {\alpha \nu } 2\int_0^te^{ \left(2-\frac{3\alpha} 4\right )\nu  \lambda_1(s-t)+2\|\nabla h\|_{L^\infty}\int_{s}^t|z(\theta_\tau\omega)|\, \d \tau}\|A^{\frac12}v(s)\|^2\, \d s\nonumber\\
&\quad \leq e^{-\left(2-\frac{3\alpha} 4\right ) \nu  \lambda_1t+2\|\nabla h\|_{L^\infty}\int_{0}^t|z(\theta_\tau\omega)|\, \d \tau}\|v(0)\|^2  \nonumber \\
& \qquad
+C\int_0^te^{ \left(2-\frac{3\alpha} 4\right ) \nu \lambda_1(s-t)+2\|\nabla h\|_{L^\infty}\int_{s}^t|z(\theta_\tau\omega)|\, \d \tau}  \left (1+|z(\theta_s\omega)|^4 \right) \d s, \nonumber
\end{align}
and then
\begin{align}
&\|v(t)\|^2+\frac {\alpha \nu} 2\int_0^te^{ \left(2-\frac{3\alpha} 4\right ) \nu \lambda_1(s-t) }\|A^{\frac12}v(s)\|^2\, \d s\nonumber\\
&\quad \leq e^{-\left(2-\frac{3\alpha} 4\right ) \nu \lambda_1t+2\|\nabla h\|_{L^\infty}\int_{0}^t|z(\theta_\tau\omega)|\, \d \tau}\|v(0)\|^2  \nonumber \\
& \qquad
+C\int_0^te^{ \left(2-\frac{3\alpha} 4\right ) \nu \lambda_1(s-t)+2\|\nabla h\|_{L^\infty}\int_{s}^t|z(\theta_\tau\omega)|\, \d \tau}  \left (1+|z(\theta_s\omega)|^4 \right)  \d s. \label{sep6.5}
\end{align}

Note that by the ergodicity  \ceqref{erg} for $m=1$, we have
\[
\lim\limits_{t\rightarrow \pm \infty}\frac{1}{t}\int_0^t|z(\theta_s\omega)| \, \d s=\mathbb{E}(|z(\theta_t\omega)| )=\frac{ 1}{\sqrt{\pi}}
, \quad \omega\in \Omega,
\]
so    there exists a  random variable  $ T_1(\omega)\geq1$ such that, for all $t\geq T_1(\omega)$,
 \[
 \frac 1t \int_{0}^t|z(\theta_s\omega)|\, \d s
 \leq  \frac{  1 }{\sqrt{\pi}} + \frac{\beta}{\sqrt{\pi}}     ,
\]
and then
\begin{align}
2\|\nabla h\|_{L^\infty}\int_{0}^t|z(\theta_s\omega)|\, \d s &\leq
   \frac{2\|\nabla h\|_{L^\infty}  }{ \sqrt{\pi}}  \left( 1 +  \beta  \right ) t  \nonumber \\
   &= \left (2 \nu \lambda_1 -  \alpha \nu  \lambda_1 \right) t \quad \text{(by  \ceqref{c2})} .\label{erg1}
\end{align}
Hence,  by   \ceqref{sep6.5} and  \ceqref{erg1},
\begin{align}
&\|v(t)\|^2+\frac {\alpha \nu }2\int_0^te^{ \left(2-\frac{3\alpha} 4\right )\nu  \lambda_1(s-t) }\|A^{\frac12}v(s)\|^2\, \d s\nonumber\\
&\quad \leq e^{ - \frac{\alpha} 4  \nu  \lambda_1t }\|v(0)\|^2
+C\int_0^te^{ \left(2-\frac{3\alpha} 4\right )\nu  \lambda_1(s-t)+2\|\nabla h\|_{L^\infty}\int_{s}^t|z(\theta_\tau\omega)|\, \d \tau}  \left (1+|z(\theta_s\omega)|^4 \right)   \d s. \nonumber
\end{align}
With   $\lambda:= \alpha \nu    \lambda_1 /4, $
 this estimate is then rewritten as
\begin{align}
&\|v(t)\|^2+\frac {\alpha \nu } 2\int_0^te^{ \left( \frac 8 \alpha -3  \right) \lambda (s-t)}\|A^{\frac12}v(s)\|^2\, \d s\nonumber\\
 &\quad\leq e^{-\lambda t}\|v(0)\|^2
 +C\int_0^te^{ \left( \frac 8 \alpha -3  \right) \lambda (s-t)+2\|\nabla h\|_{L^\infty}\int_{s}^t|z(\theta_\tau\omega)|\, \d \tau}  \left (1+|z(\theta_s\omega)|^4 \right)  \d s. \nonumber
\end{align}
Replacing $\omega$ with $\theta_{-t}\omega$ yields
 \begin{align}
&\|v(t, \theta_{-t}\omega, v(0))\|^2+\frac {\alpha  \nu }2\int_0^te^{ \left( \frac 8 \alpha -3  \right) \lambda (s-t)}\|A^{\frac12}v(s, \theta_{-t}\omega, v(0))\|^2\, \d s\nonumber\\
 &\quad\leq e^{-\lambda t}\|v(0)\|^2
 +C\int^0_{-t} e^{ \left( \frac 8 \alpha -3  \right) \lambda  s +2\|\nabla h\|_{L^\infty}\int_{s}^0 |z(\theta_\tau\omega)|\, \d \tau}  \left (1+|z(\theta_s\omega)|^4 \right) \d s. \nonumber
\end{align}
Define a random variable by
\ben
 \zeta_1(\omega) :=C\int_{-\infty}^0  e^{ \left( \frac 8 \alpha -3  \right) \lambda  s +2\|\nabla h\|_{L^\infty}\int_{s}^0 |z(\theta_\tau\omega)|\, \d \tau}  \left (1+|z(\theta_s\omega)|^4 \right)   \d s,\quad \omega\in \Omega .
\ee
Then $\zeta_1(\cdot)$ is a tempered random variable such that  $\zeta_1(\omega)\geq 1$ and
\begin{align}
&\|v(t,\theta_{-t} \omega, v(0))\|^2+ \frac {\alpha \nu } 2\int_0^t e^{ \left ( \frac{8}{\alpha} -3  \right) \lambda (s-t)}\|A^{\frac12}v(s, \theta_{-t}\omega, v(0))\|^2\, \d s \nonumber \\
 &\quad \leq e^{-\lambda t}\|v(0)\|^2
 + \zeta_1(\omega) ,\quad t\geq T_1(\omega).  \label{4.29}
\end{align}
The lemma follows.
\end{proof}

\begin{lemma}[$H^1$ bound] \label{lem:H1bound}
Let  $f\in H$ and  \cref{assum} hold.
Then for any bounded set $ B$ in $ H$ there is a random variable $ T_B(\omega)>T_1(\omega)$ such that,  for all $t\geq T_B(\omega)$,
\[
\sup_{v(0)\in B} \left( \big\|A^{\frac12}v(t,\theta_{-t}\omega,v(0))  \big\|^2 + \int_{t-\frac 12 }^t \|Av(s,\theta_{-t}\omega,v(0))\|^2  \ \d s\right)
\leq  \zeta_2(\omega),
\]
where $\zeta_2(\omega)$  is a tempered random variable given by  \ceqref{mar8.8} such that $\zeta_2(\omega) > \zeta_1(\omega)\geq 1$, $ \omega\in\Omega$.
\end{lemma}
\begin{proof}
Multiplying   \ceqref{2.2} by $Av$ and integrating over $\mathbb T^2$, by integration by parts we have
\begin{align}\label{4.6}
\frac12\frac{\d}{\d t}\|A^{\frac12}v\|^2+ \nu \|Av\|^2&=-\big(B(v+hz(\theta_t\omega),v+hz(\theta_t\omega)),\, Av \big)\nonumber\\
&\quad \ +\big(f-\nu Ahz(\theta_t\omega) +hz(\theta_t\omega),Av \big).
\end{align}
 Applying the H\"{o}lder inequality,  the Gagliardo-Nirenberg inequality, Poincar\'{e}'s inequality and the Young inequality, we obtain
\begin{align}\label{4.4}
&\big| \big(B(v+hz(\theta_t\omega),v+hz(\theta_t\omega)),\, Av \big)\big|
= \big| \big(B(v+hz(\theta_t\omega),v+hz(\theta_t\omega)),Ahz(\theta_t\omega) \big) \big| \nonumber\\
& \quad \leq \|Ahz(\theta_t\omega)\|\|v+hz(\theta_t\omega)\|_{L^4}\|A^{\frac12}(v+hz(\theta_t\omega))\|_{L^4}\nonumber\\
& \quad \leq C\|Ahz(\theta_t\omega)\|\|v+hz(\theta_t\omega)\|^{\frac12}\|A^{\frac12}(v+hz(\theta_t\omega))\|
\|A(v+hz(\theta_t\omega))\|^{\frac12}\nonumber\\
& \quad \leq C\|Ahz(\theta_t\omega)\|\|A^{\frac12}(v+hz(\theta_t\omega))\|^{\frac32}
\|A(v+hz(\theta_t\omega))\|^{\frac12}\nonumber\\
& \quad \leq  C\|Ahz(\theta_t\omega)\| \left (\|A^{\frac12}v\|^{\frac32}+\|A^{\frac12}hz(\theta_t\omega)\|^{\frac32}\right) \left
(\|Av\|^{\frac12}+\|Ahz(\theta_t\omega)\|^{\frac12} \right)\nonumber\\
& \quad \leq \frac  \nu 4\|Av\|^2+C\left (1+\|Ah\|^2|z(\theta_t\omega)|^2 \right )\|A^{\frac12}v\|^2+C\left (1+|z(\theta_t\omega)|^6\right),
\end{align}
and
\begin{align}
\big (f-\nu Ahz(\theta_t\omega) +hz(\theta_t\omega),Av \big)
& \leq \frac  \nu 4 \|Av\|^2+C\|f\|^2+ C  |z(\theta_t\omega)|^2   .\label{4.5}
\end{align}
Substituting  \ceqref{4.4} and  \ceqref{4.5} into  \ceqref{4.6} we have
\begin{align}\label{4.27}
\frac{\d}{\d t}\|A^{\frac12}v\|^2 +  \nu  \|Av\|^2
&\leq  C\left ( 1 + \|Ah\|^2|z(\theta_t\omega)|^2 \right)\|A^{\frac12}v\|^2  +
C \left(1 +|z(\theta_t\omega)|^6+\|f\|^2 \right)\nonumber \\
&\leq  C\left ( 1 +  |z(\theta_t\omega)|^2 \right)  \|A^{\frac12}v\|^2  +
C \left(1 +|z(\theta_t\omega)|^6  \right).
\end{align}
Applying Gronwall's lemma on $(\eta, t)$ with $\eta\in  (t-1,t )$, $t >1$, we have
\be
&\|A^{\frac12}v(t)\|^2
+ \nu \int_\eta ^te^{ \int_s^tC(1+|z(\theta_\tau\omega)|^2)\, \d \tau}\|Av(s)\|^2 \, \d s \nonumber\\
&\quad \leq e^{ \int_\eta^tC(1+|z(\theta_\tau\omega)|^2)\, \d \tau}\|A^{\frac12}v( \eta)\|^2 +C\int_\eta^te^{ \int_s^t C(1+ |z (\theta_\tau\omega)|^2)\, \d \tau} \left (1 +|z(\theta_s\omega)|^6 \right)\d s,
\ee
and then integrating over $\eta\in (t-1,t-\frac 12) $ yields
\be
& \frac12  \|A^{\frac12}v(t )\|^2
+\frac  \nu 2\int_{t-\frac 12 }^t e^{ \int_s^tC(1+|z(\theta_\tau\omega)|^2)\, \d \tau}\|Av(s)\|^2\, \d s \nonumber\\
&\quad \leq \int_{t-1}^{t-\frac 12}  e^{ \int_\eta^tC(1+|z(\theta_\tau\omega)|^2)\, \d \tau} \|A^{\frac12}v( \eta)\|^2 \, \d \eta
 \nonumber\\
&\qquad +C\int_{t-1}^te^{ \int_s^t C(1+ |z (\theta_\tau\omega)|^2)\, \d \tau} \left (1 +|z(\theta_s\omega)|^6 \right) \d s \nonumber\\
&\quad \leq Ce^{   C \int_{ t-1}^t (1+|z(\theta_\tau\omega)|^2) \, \d \tau}  \left[  \int_{t-1}^{t  }  \|A^{\frac12}v( \eta)\|^2 \, \d \eta  + \int_{t-1}^t \left (1 +|z(\theta_s\omega)|^6 \right)\d s\right]  .
\ee
Replacing $\omega$ with $\theta_{-t}\omega$ we obtain
\be
&   \|A^{\frac12}v(t,\theta_{-t}\omega, v(0) )\|^2
+   \nu \int_{t-\frac 12 }^te^{ C(t-s) + \int_{s-t}^0C |z(\theta_\tau\omega)|^2 \, \d \tau}\|Av(s ,\theta_{-t}\omega, v(0))\|^2\, \d s\nonumber\\
&\quad  \leq  C e^{ C \int_{-1}^0 |z (\theta_\tau\omega)|^2\, \d \tau}
\left [ \int_{t-1}^t   \|A^{\frac12}v(s,\theta_{-t}\omega, v(0) )\|^2\, \d s+\int_{-1}^0 \left(1+ |z(\theta_{s} \omega)|^6  \right)  \d s\right].
\ee
By      \ceqref{4.29},
\be
 \nu  \int_{t-1}^t  \|A^{\frac12}v(s,\theta_{-t}\omega, v(0) )\|^2  \, \d s
 & \leq   \nu e^{\left( \frac 8\alpha -3 \right)  \lambda }   \int_{t-1}^t  e^{\left( \frac 8\alpha -3 \right)  \lambda (s-t) } \|A^{\frac12}v(s,\theta_{-t}\omega, v(0) )\|^2 \, \d s
 \\
 & \leq  \frac{2}{\alpha    }   e^{\left( \frac 8\alpha -3 \right)  \lambda }   \left( e^{-\lambda t}\|v(0)\|^2
 +   \zeta_1(\omega)\right) , \quad    t\geq T_1(\omega) , \nonumber
\ee
so
we have
\be
&   \|A^{\frac12}v(t,\theta_{-t}\omega, v(0) )\|^2
 +  \nu \int_{t-\frac 12 }^t  \|Av(s ,\theta_{-t}\omega, v(0))\|^2\, \d s \nonumber\\
 &\quad  \leq  \|A^{\frac12}v(t,\theta_{-t}\omega, v(0) )\|^2
+  \nu \int_{t-\frac 12 }^te^{ C(t-s) + \int_{s-t}^0C |z(\theta_\tau\omega)|^2 \, \d \tau}\|Av(s ,\theta_{-t}\omega, v(0))\|^2\, \d s\nonumber\\
&\quad \leq   C e^{ C \int_{-1}^0 |z (\theta_\tau\omega)|^2\, \d \tau} \left(   e^{-\lambda t}\|v(0)\|^2
 +   \zeta_1(\omega)  +
\int_{-1}^0 |z(\theta_{s} \omega)|^6  \, \d s \right) , \quad    t\geq T_1(\omega) ,
\ee
where  we have used the fact  that $\zeta_1(\omega)\geq 1$.
Therefore, by
\be \label{mar8.8}
 \zeta_2(\omega) :=   C e^{ C \int_{-1}^0 |z (\theta_\tau\omega)|^2\, \d \tau} \left(      \zeta_1(\omega)  +
\int_{-1}^0 |z(\theta_{s} \omega)|^6   \ \d s \right) ,
\quad \omega\in \Omega,
\ee
we defined a tempered random variable  and the lemma follows.
\end{proof}

We define a random set
$\mathfrak{B}=\{\mathfrak{B}(\omega)\}_{\omega\in\Omega}$ in $H^1$  by
\begin{align}\label{4.19}
\mathfrak{B}(\omega):= \left \{ v \in H^1:\|A^{\frac12}v\|^2\leq \zeta_2(\omega)
\right \}, \quad \omega\in\Omega,
\end{align}
where $\zeta_2(\cdot)$ is the tempered random variable defined by  \ceqref{mar8.8}. Then   by  \cref{lem:H1bound}  $\mathfrak{B}$ is a  random absorbing set of the RDS $\phi$ generated by the  NS equation   \ceqref{2.2}.  In addition,  the compact embedding $H^1 \hookrightarrow H$ gives the compactness of $\mathfrak B(\omega)$ in $H$, hence the RDS $\phi$ has a    random attractor $\A$ in $H$. In addition, in view of Langa \& Robinson \cite{langa2}, this random attractor has finite fracal dimension in $H$.  Summarizing,

\begin{theorem}\label{theorem4.6}  Let  \cref{assum} hold and $f\in H$. Then the  RDS $\phi$ generated by  the random NS equations  \ceqref{2.2} has a  random absorbing set $\mathfrak{B}$ which is a tempered and bounded random set in $H^1 $, and has also a   random attractor $\mathcal A $ in $H$. In addition, $\mathcal A$ has finite fractal dimension in $H$:
\[
d_f^H(\A(\omega))  \leq d,\quad \omega \in \Omega,
\]
for some positive constant $d$.
\end{theorem}

 In the sequel we shall show that this attractor $\A$ is in fact an $(H,H^2)$-random attractor of $\phi$.

\section{Construction of  an $H^2$ random absorbing set}\label{sec4}

In this section we shall construct an  $H^2$ random absorbing set of  \ceqref{2.2}.   This will be done by estimating  the difference between the   solutions of the random  equation  \ceqref{2.2} and that of the  deterministic  equation  \ceqref{2.1} within the global attractor $\A_0$, see  \cref{rmk1}.  This comparison approach seems first employed in Cui \& Li \cite{cui}.

\subsection{$H^2$-distance between random and deterministic trajectories}

Since  \cref{theorem4.6}  shows that the random NS equation  \ceqref{2.2}  has an $H^1$ random absorbing set $\mathfrak B$ and  \cref{lem:det} shows that  the deterministic equation   \ceqref{2.1} has a global attractor $\mathcal A_0$ bounded in $H^2$, it suffices to restrict ourselves to the random absorbing set  $\mathfrak B$ and the  global attractor $\mathcal A_0$.
  Note that    the $\omega$-dependence  of each  random time in the following estimates will be crucial for later analysis.

\begin{lemma}[$H^1$-distance]
\label{lemma4.3}
 Let  \cref{assum} hold and $f\in H$. Then there   exist   random variables $T_{\mathfrak B}(\cdot)$ and    $\zeta_4(\cdot)$, where  $\zeta_4(\cdot)$ is tempered,  such that the solutions $v$ of the random NS equations  \ceqref{2.2} and $u$  of the deterministic  equations   \ceqref{2.1} satisfy
\begin{align}
 \big \|A^{\frac 12} v(t,\theta_{-t}\omega,v(0))-A^{\frac 12} u(t,u(0)) \big\| ^2 \leq \zeta_4(\omega),\quad \forall t\geq T_{\mathfrak B}(\omega), \nonumber
\end{align}
for all $v(0)\in\mathfrak{B}(\theta_{-t}\omega)$
 and $u(0)\in\mathcal{A}_0$ for $\omega\in \Omega$.
\end{lemma}
\begin{proof}
Let  $w=v-u$ be the difference between the two solutions. Then it satisfies
\begin{align}\label{4.8}
\frac{\d w}{\d t}+ \nu  Aw+B( v+hz(\theta_t\omega))-B(u) = hz(\theta_t\omega) - \nu Ahz(\theta_t\omega).
\end{align}
Multiplying   \ceqref{4.8} by $Aw$ and integrating over $\mathbb T^2$, by integration by parts we have
\begin{align}\label{4.11}
 \frac12\frac{\d}{\d t}\|A^{\frac12}w\|^2+ \nu \|Aw\|^2
&= \big (hz(\theta_t\omega)-\nu Ahz(\theta_t\omega) ,Aw \big )  - \big(B(v+hz(\theta_t\omega))-B(u),Aw\big) \nonumber \\
&= \big  (hz(\theta_t\omega)-\nu Ahz(\theta_t\omega) ,Aw \big ) - \big (B(w+hz(\theta_t\omega),u),Aw \big ) \nonumber\\
&\quad -  \big (B(v+hz(\theta_t\omega),w+hz(\theta_t\omega)),Aw  \big ) \nonumber \\
&=: I_{1}(t) + I_{2}(t)+I_{3}(t) .
\end{align}
Then applying H\"{o}lder's inequality, the Gagliardo-Nirenberg inequality,  Poincar\'{e}'s inequality and Young's inequality, we obtain the estimates
\begin{align}\label{4.10}
 I_{1}(t)   &\leq \nu\|Ahz(\theta_t\omega)\|\|Aw\|+\|hz(\theta_t\omega)\|\|Aw\|\nonumber\\
&\leq \frac  \nu 4\|Aw\|^2+C|z(\theta_t\omega)|^2,
\end{align}
\begin{align}
I_{2}(t)&\leq \|w+hz(\theta_t\omega)\|_{L^\infty}\|A^{\frac12}u\|\|Aw\|\nonumber\\
&\leq C\|A^{\frac12}u\|\|Aw\| \left(\|Aw\|^{\frac12}+\|Ahz(\theta_t\omega)\|^{\frac12} \right) \left
(\|A^{\frac12}w\|^{\frac12}+\|A^{\frac12}hz(\theta_t\omega)\|^{\frac12} \right )\nonumber\\
&\leq C\|A^{\frac12}u\|\|Aw\|^{\frac32}\|A^{\frac12}w\|^{\frac12}
+C\|A^{\frac12}u\|\|Aw\|^{\frac32}|z(\theta_t\omega)|^{\frac12}\nonumber\\
&\quad +C\|A^{\frac12}u\|\|Aw\|\|A^{\frac12}w\|^{\frac12}|z(\theta_t\omega)|^{\frac12}+
C\|A^{\frac12}u\|\|Aw\||z(\theta_t\omega)|\nonumber\\
&\leq \frac  \nu 8\|Aw\|^2+C\|A^{\frac12}u\|^4\|A^{\frac12}w\|^2
+C \left (\|A^{\frac12}u\|^2+\|A^{\frac12}u\|^4 \right) |z(\theta_t\omega)|^2,
\end{align}
and
\begin{align}
I_{3}(t)&\leq \|v+hz(\theta_t\omega)\|_{L^4}\|A^{\frac12}(w+hz(\theta_t\omega))\|_{L^4}\|Aw\|\nonumber\\
&\leq C \left (\|v\|_{L^4}+\|hz(\theta_t\omega)\|_{L^4} \right ) \left(\|A^{\frac12}w\|_{L^4}+\|A^{\frac12}hz(\theta_t\omega)\|_{L^4} \right)\|Aw\|\nonumber\\
&\leq C\|A^{\frac12}v\|\|A^{\frac12}w\|^{\frac12}\|Aw\|^{\frac32}
+C\|A^{\frac12}hz(\theta_t\omega)\|\|A^{\frac12}w\|^{\frac12}\|Aw\|^{\frac32}\nonumber\\
&\quad +C\|A^{\frac12}v\|\|Ahz(\theta_t\omega)\|\|Aw\|+C\|A^{\frac12}hz(\theta_t\omega)\|\|Ahz(\theta_t\omega)\|\|Aw\|
\nonumber\\
&\leq \frac \nu 8\|Aw\|^2+C \! \left (\|A^{\frac12}v\|^4+|z(\theta_t\omega)|^4 \right )\|A^{\frac12}w\|^2
+C \! \left (\|A^{\frac12}v\|^4+|z(\theta_t\omega)|^4 \right).
\nonumber
\end{align}
 Since $u$ is a trajectory within the global attractor $\mathcal{A}_0$, we have $\|A^{\frac12}u(t)\|\leq C$ for all $t\geq0$, so
\begin{align}\label{4.9}
I_{2}(t)
\leq \frac \nu 8\|Aw\|^2+C\|A^{\frac12}w\|^2+C|z(\theta_t\omega)|^2.
\end{align}
Substituting  \ceqref{4.10}- \ceqref{4.9} into  \ceqref{4.11} we have
\begin{align}
\frac{\d}{\d t}\|A^{\frac12}w\|^2+  \nu \|Aw\|^2
&\leq C \left (1+\|A^{\frac12}v\|^4+|z(\theta_t\omega)|^4 \right )\|A^{\frac12}w\|^2  \nonumber \\
&\quad
+C \left(1+\|A^{\frac12}v\|^4+|z(\theta_t\omega)|^4\right) .
\label{sep6.1}
\end{align}

For $t >1$ and $s\in[t-1,t]$, applying Gronwall's lemma to  \ceqref{sep6.1} we obtain
\begin{align}
\|A^{\frac12}w(t)\|^2&\leq e^{\int_s^tC(1+\|A^{\frac12}v(\eta) \|^4+|z(\theta_\eta\omega)|^4)\, \d \eta}\|A^{\frac12}w(s)\|^2\nonumber\\
&
 \quad +C\int_s^te^{\int_\tau^tC(1+\|A^{\frac12}v(\eta) \|^4+|z(\theta_\eta\omega)|^4)\, \d \eta} \left(1+\|A^{\frac12}v(\tau)\|^4+|z(\theta_\tau\omega)|^4 \right) \d \tau\nonumber\\
&\leq Ce^{\int_s^tC(\|A^{\frac12}v(\eta) \|^4+|z(\theta_\eta\omega)|^4)\, \d \eta} \nonumber  \\
&\quad \times
 \left[\|A^{\frac12}w(s)\|^2
+\int_{t-1}^t    \left(1+\|A^{\frac12}v\|^4+|z(\theta_\tau\omega)|^4\right) \d \tau \right]. \nonumber
\end{align}
Integrating w.r.t$.$  $s$ over $[t-1,t]$ yields
\begin{align*}
\|A^{\frac12}w(t)\|^2
&\leq Ce^{\int_{t-1}^tC(\|A^{\frac12}v(\eta)\|^4+|z(\theta_\eta\omega)|^4)\, \d \eta}\nonumber\\
&\quad \times \left[\int_{t-1}^t\|A^{\frac12}w(s)\|^2\, \d s
+\int_{t-1}^t  \left(1+\|A^{\frac12}v\|^4+|z(\theta_\tau \omega)|^4\right) \d \tau\right]. \nonumber
\end{align*}
   Since $ u$ is a trajectory within the global attractor, $\|A^{\frac 12} u\|^2$ is uniformly bounded. Therefore,
\begin{align*}
\|A^{\frac 12} w(s)\|^2
 =\|A^{\frac 12} (v-u) (s)\|^2
  \leq \|A^{\frac 12} v (s)\|^4 + C,
\end{align*}
and then
\ben
\|A^{\frac12}w(t)\|^2
 \leq Ce^{\int_{t-1}^tC(\|A^{\frac12}v(\eta)\|^4+|z(\theta_\eta\omega)|^4)\, \d \eta}  \int_{t-1}^t  \left(1+\|A^{\frac12}v(\tau)\|^4+|z(\theta_\tau\omega)|^4 \right)\d \tau .
\ee
Replacing $\omega$ with $\theta_{-t}\omega$, for $t>1$  we deduce that
\begin{align}\label{4.14}
  \|A^{\frac12}w(t,\theta_{-t}\omega,w(0))\|^2
&\leq Ce^{\int_{t-1}^tC \|A^{\frac12}v(\eta,\theta_{-t}\omega,v(0))\|^4 \d \eta+  C\int_{-1}^0|z(\theta_{\eta}\omega)|^4 \, \d \eta}  \nonumber\\
 &\quad
 \times \left(\int_{t-1}^t  \|A^{\frac12}v(s,\theta_{-t}\omega,v(0))\|^4 \, \d s + \int_{-1}^0  |z(\theta_{s}\omega)|^4 \, \d s +1 \right) .
\end{align}

Recall from  \ceqref{4.27} that
\[
\frac{\d}{\d t}\|A^{\frac12}v\|^2 + \nu \|Av\|^2
 \leq  C\left ( 1 +  |z(\theta_t\omega)|^2 \right)  \|A^{\frac12}v\|^2  +
C \left(1 +|z(\theta_t\omega)|^6  \right).
\]
Then for $t>3$  integrating both sides over $(\tau, t)$ with $\tau\in (t-3,t-2)$ we obtain
\ben
&
\|A^{\frac12}v (t) \|^2 -\|A^{\frac12}v(\tau) \|^2  + \nu \int^t_{t-2} \|A v(s)\|^2 \, \d s \\
&\quad
 \leq  C \int^t_{t-3} \left(1+  |z(\theta_s\omega)|^2 \right) \|A^{\frac12}v(s) \|^2\, \d s
+ C\int^t_{t-3} \left(1 +|z(\theta_s\omega)|^6 \right )  \d s,
\ee
and then integrating over $\tau\in (t-3,t-2)$ yields
\ben
\|A^{\frac12}v (t) \|^2+  \nu \int^t_{t-2} \|A v(s)\|^2\, \d s
& \leq   C \int^t_{t-3} \left (|z(\theta_s\omega)|^2+1\right)  \|A^{\frac12}v(s) \|^2 \, \d s \\
&\quad + C\int^t_{t-3} \left (1 +|z(\theta_s\omega)|^6 \right)  \d s ,\quad t> 3.
\ee
For any $\varepsilon\in  [0,2]$, replacing $\omega$ with $\theta_{-t-\varepsilon} \omega$  we then have
\begin{align}
 & \|A^{\frac12}v (t,\theta_{-t-\varepsilon} \omega, v(0)) \|^2 +  \nu \int_{t-2}^t  \|Av (s,\theta_{-t-\varepsilon} \omega, v(0)) \|^2 \, \d s \nonumber \\
& \quad \leq   C \int^t_{t-3} \! \left (|z(\theta_{s-t-\varepsilon} \omega)|^2+1\right)  \|A^{\frac12}v(s,\theta_{-t-\varepsilon} \omega, v(0)) \|^2 \, \d s + C\int^{-\varepsilon}_{-\varepsilon-3} \! \left(1 +|z(\theta_s\omega)|^6 \right )  \d s\nonumber\\
& \quad \leq   C\left (  \sup_{\tau\in (-5,0)}|z(\theta_\tau\omega)|^2+1\right) \int^{t+\varepsilon}_{t+\varepsilon-5}  \|A^{\frac12}v(s,\theta_{-t-\varepsilon} \omega, v(0)) \|^2 \, \d s \nonumber\\
&\qquad + C\int^{0}_{-5}    \left(1 +|z(\theta_s\omega)|^6 \right)   \d s ,\quad t>5 . \label{mar8.1}
\end{align}
By  \ceqref{4.29},
$$
  \frac {\alpha  \nu }2\int_0^te^{ \left( \frac 8 \alpha -3  \right) \lambda (s-t)}\|A^{\frac12}v(s, \theta_{-t}\omega, v(0))\|^2\, \d s
 \leq e^{-\lambda t}\|v(0)\|^2
 + \zeta_1(\omega) ,\quad t\geq T_1(\omega) ,  $$
where $\zeta_1(\omega)\geq 1$ is a tempered random variable,
so we obtain
 \ben
  \int^{t+\varepsilon}_{t+\varepsilon-5}  \|A^{\frac12}v(s,\theta_{-(t+\varepsilon)} \omega, v(0)) \|^2 \, \d s
  &\leq  C \int^{t+\varepsilon}_{t+\varepsilon-5}  e^{ \left( \frac 8 \alpha -3  \right) \lambda (s-t)}\|A^{\frac12}v(s,\theta_{-t-\varepsilon} \omega, v(0)) \|^2 \, \d s \\[0.8ex]
  & \leq Ce^{-\lambda t}\|v(0)\|^2
 + C \zeta_1(\omega) ,\quad t\geq T_1(\omega)+5 ,
 \ee
uniformly for $\varepsilon \in [0,2]$.  Therefore, coming back to   \ceqref{mar8.1} we obtain
\begin{align}
 &\sup_{\varepsilon \in [0,2]} \left(\|A^{\frac12}v (t,\theta_{-t-\varepsilon} \omega, v(0)) \|^2  +  \nu \int_{t-2}^t  \|Av (s,\theta_{-t-\varepsilon} \omega, v(0)) \|^2 \, \d s\right ) \nonumber \\
&\quad \leq   C\left (  \sup_{\tau\in (-5,0)}|z(\theta_\tau\omega)|^2+1\right)\left(e^{-\lambda t}\|v(0)\|^2
 +   \zeta_1(\omega)   \right) + C\int^{0}_{-5} \left (1 +|z(\theta_s\omega)|^6 \right )    \d s \nonumber \\
 & \quad \leq   C\left (  \sup_{\tau\in (-5,0)}|z(\theta_\tau\omega)|^6+1\right)\left(e^{-\lambda t}\|v(0)\|^2
 +   \zeta_1(\omega)   \right)  , \quad t\geq T_1(\omega)+5.  \nonumber
\end{align}
 In other words, this means that, for   $t\geq T_1(\omega)+7$ and $\eta\in [t-2,t ]$ $($so that $\eta\geq T_1(\omega)+5),$
\ben
& \|A^{\frac12}v (\eta ,\theta_{-t } \omega, v(0)) \|^2 + \nu \int_{\eta-2}^\eta  \|Av (s,\theta_{-t} \omega, v(0)) \|^2 \, \d s\\
  & \quad =  \|A^{\frac12}v (\eta ,\theta_{-\eta-(t-\eta) } \omega, v(0)) \|^2
  + \nu  \int_{\eta-2}^\eta  \|Av (s,\theta_{-\eta-(t-\eta)} \omega, v(0)) \|^2 \, \d s \\
&\quad \leq    C\left (  \sup_{\tau\in (-5,0)}|z(\theta_\tau\omega)|^6+1\right)\left(e^{-\lambda  \eta}\|v(0)\|^2
 +   \zeta_1(\omega)   \right)
  \\
&\quad  \leq   C\left (  \sup_{\tau\in (-5,0)}|z(\theta_\tau\omega)|^6+1\right)\left(e^{-\lambda t}\|v(0)\|^2
 +   \zeta_1(\omega)   \right) ,
\ee
where in the last inequality we have used $\eta\geq t-2$.
Notice that the initial value $v(0)$ comes from within the absorbing set $\mathfrak B$, so there is a random variable $T_{\mathfrak B}(\omega) \geq T_1(\omega)+ 7$
such that
\be \label{timeB}
 \sup_{v(0)\in \mathfrak B(\theta_{-t} \omega)}  e^{-\lambda t} \|v(0)\|^2 \leq 1,\quad t\geq T_{\mathfrak B}(\omega).
\ee
Therefore, from the above estimate it follows  the crucial estimate
\begin{align}
 &\sup_{\eta\in [t-2,t]}
  \left(
 \|A^{\frac12}v (\eta ,\theta_{-t } \omega, v(0)) \|^2
 + \nu \int_{\eta-2}^\eta  \|Av (s,\theta_{-t} \omega, v(0)) \|^2 \, \d s\right) \notag  \\
& \quad \leq  C\left (  \sup_{\tau\in (-5,0)}|z(\theta_\tau\omega)|^6+1\right)\big(1
 +   \zeta_1(\omega)   \big)
 =:  \zeta_3(\omega) ,\quad t\geq T_{\mathfrak B} (\omega) ,\label{mar8.3}
\end{align}
where  by   $ \zeta_3 $ we defined a tempered random variable.
An immediate  consequence is the  estimate
\be \label{mar8.2}
  \int_{t-1}^t\|A^{\frac12}v(\eta,\theta_{-t}\omega,v(0))\|^4 \, \d \eta
&  \leq    |\zeta_3(\omega)|^2 ,  \quad t\geq T_{\mathfrak B} (\omega) .
\ee

  Substituting  \ceqref{mar8.2} into  \ceqref{4.14}, we obtain
\begin{align}  \label{mar8.4}
\|A^{\frac12}w(t,\theta_{-t}\omega,w(0))\|^2
&\leq Ce^{ |\zeta_3 (\omega)|^2+ C\int_{-1}^0|z(\theta_{\eta}\omega)|^4   \d \eta}  \left( |\zeta_3(\omega)|^2 + \int_{-1}^0 |z(\theta_{s}\omega)|^4 \,  \d s  \right) \nonumber \\
&\leq Ce^{ 2|\zeta_3 (\omega)|^2+ C\int_{-1}^0|z(\theta_{\eta}\omega)|^4   \d \eta}
=: \zeta_4(\omega) , \quad t\geq T_{\mathfrak B} (\omega) .
\nonumber
\end{align}
Then we have the lemma.
\end{proof}

\begin{lemma}[$H^2$-distance] \label{lem:H2bd}  Let  \cref{assum} hold and $f\in H$.
Then there is a tempered random variable $\zeta_5(\cdot)$ such that the solutions $v$ of the random equations  \ceqref{2.2} and $u$  of the deterministic equations   \ceqref{2.1} satisfy
\[
 \|Av(t,\theta_{-t}\omega,v(0))-Au(t,u(0))\|^2\leq \zeta_5(\omega),
\quad t\geq T_{\mathfrak B}(\omega),
\]
for  $v(0)\in\mathfrak{B}(\theta_{-t}\omega)$
 and $u(0)\in\mathcal{A}_0$,
where $T_{\mathfrak B}$ is the random variable given in   \cref{lemma4.3}.
\end{lemma}
\begin{proof}

Multiplying   \ceqref{4.8} by $A^2w$ and integrating over $\mathbb T^2$, by integration by parts we have
\begin{align}\label{4.17}
 \frac12\frac{\d}{\d t}\|Aw\|^2+  \nu  \|A^{\frac32}w\|^2
&= \left(-\nu Ahz(\theta_t\omega) +hz(\theta_t\omega),A^2w \right) \nonumber \\
&\quad -\left(B(v+hz(\theta_t\omega)) -B(u),A^2w \right)\nonumber\\
&=  \left(-\nu Ahz(\theta_t\omega) +hz(\theta_t\omega),A^2w \right)\nonumber\\
&\quad -\left(B(v+hz(\theta_t\omega),w+hz(\theta_t\omega)),A^2w \right )\nonumber\\
&\quad
-\left(B(w+hz(\theta_t\omega),u),A^2w \right )\nonumber\\
&=: I_{4}(t)+I_{5}(t)+I_{6}(t).
\end{align}
Now we estimate  \ceqref{4.17} term by term to obtain  \ceqref{sep6.3}.
Applying H\"{o}lder's inequality, Young's inequality and the Poincar\'e inequality  \ceqref{poin},  we obtain, since $h\in H^3$,
\begin{align}
I_{4}(t)&\leq  \nu \|A^{\frac32}hz(\theta_t\omega)\|\|A^{\frac32}w\|+ \|A^{\frac12}hz(\theta_t\omega)\|\|A^{\frac32}w\|\nonumber\\
&\leq \frac{ \nu }{8}\|A^{\frac32}w\|^2+C  |z(\theta_t\omega)|^2. \label{4.18}
\end{align}
Applying H\"{o}lder's inequality and the Gagliardo-Nirenberg inequality, we obtain
\begin{align}
I_{5}(t)&\leq\|A^\frac12u_h\|_{L^4}\|\nabla(w+hz(\theta_t\omega))\|_{L^4}\|A^\frac32w\|
+\|u_h\|_{L^4}\|A(w+hz(\theta_t\omega))\|_{L^4}\|A^\frac32w\|
\nonumber\\
&\leq C\|A^\frac12u_h\|^\frac12\|Au_h\|^\frac12
\|\nabla(w+hz(\theta_t\omega))\|^\frac12\|A(w+hz(\theta_t\omega))\|^\frac12
\|A^\frac32w\|\nonumber\\
& \quad +C\|u_h\|^\frac12\|A^{\frac12}u_h\|^\frac12
\|A(w+hz(\theta_t\omega))\|^\frac12\|A^{\frac32}(w+hz(\theta_t\omega))\|^\frac12\|A^\frac32w\|\nonumber\\
&\leq  C\|A^\frac12u_h\|^\frac12\|Au_h\|^\frac12
\left(\|A^{\frac12}w\|^\frac12+ |z(\theta_t\omega)|^\frac12 \right) \left (\|Aw\|^\frac12+ |z(\theta_t\omega) |^\frac12 \right)
\|A^\frac32w\|\nonumber\\
&\quad +C\|A^{\frac12}u_h\| \left (\|Aw\|^\frac12+ |z(\theta_t\omega)|^\frac12\right)\left(\|A^{\frac32}w\|^\frac12
+ | z(\theta_t\omega)|^\frac12 \right)\|A^\frac32w\|\nonumber\\
&=: J_1(t) +J_2(t) . \nonumber
\end{align}

Since $\mathcal{A}_0$ is bounded in $H^2$,  $\|Au(t) \|^2+\|A^{1/2}u(t) \|^2\leq C$ for any $t\geq 0$, so
\begin{align}\label{4.15}
\|Au_h\|^2&\leq \|Aw\|^2+\|Au\|^2
 +  \|Ahz(\theta_t\omega)\|^2 \nonumber \\
&
\leq \|Aw\|^2+  C  +C \|Ahz(\theta_t\omega)\|^2
,
\end{align}
and similarly
\begin{align}  \label{mar5.3}
  \|A^{\frac 12}w\|^2
   &\leq  \|A^{\frac 12} v\|^2 +C
   \nonumber  \\
   &\leq \|A^{\frac 12} u_h\|^2 + \|A^{\frac 12} hz(\theta_t\omega) \|^2 +C .
\end{align}
Therefore,  for  the term $J_1(t)$, from the inequalities
\be
  &C\|A^\frac12u_h\|^\frac12\|Au_h\|^\frac12\|A^\frac12w\|^\frac12
\|Aw\|^\frac12\|A^\frac32w\|
& \nonumber\\
&\quad  \leq    \frac \nu {32}\|A^\frac32w\|^2  + C \|A^\frac12u_h\|  \|A u_h\|   \|A^{\frac 12} w\|  \|Aw\|
\nonumber
\\
&\quad  \leq   \frac \nu {32}\|A^\frac32w\|^2
+ C \|A^\frac12u_h\|^2  \|A u_h\|^2+      \|A^{\frac 12} w\|^2 \|Aw\|^2
\nonumber
\\
&\quad  \leq    \frac \nu {32} \|A^\frac32w\|^2
+ C\left( \|A^\frac12u_h\|^2 +  \|A^\frac 12 w   \|^2  \right) \|A w\|^ 2
+ C  \|A^\frac12u_h\|^2   \left(  1+ |z(\theta_t\omega) |^2  \right) ,
\ee
\ben
  &C\|A^\frac12u_h\|^\frac12\|Au_h\|^\frac12\|A^\frac12w\|^\frac12
 |z(\theta_t\omega)|^\frac12\|A^\frac32w\|
& \nonumber\\
&\quad  \leq   \frac \nu {32}\|A^\frac32w\|^2  + C \|A^\frac12u_h\|  \|A u_h\|   \|A^{\frac 12} w\| |z(\theta_t\omega) |
\nonumber
\\
& \quad  \leq    \frac \nu {32}\|A^\frac32w\|^2
+ C \|A^\frac12u_h\|^2  \|A u_h\|^2+   C   \|A^{\frac 12} w\|^2 |z(\theta_t\omega) |^2
\nonumber
\\
&\quad \leq      \frac \nu {32}\|A^\frac32w\|^2
+ C  \|A^\frac12u_h\|^2  \|A w\|^2  + C \big(  \|A^\frac12u_h\|^2 +|z(\theta_t\omega) |^2 +1\big) \big(  1+|z(\theta_t\omega) |^2  \big)  ,
\ee

\be
  &C\|A^\frac12u_h\|^\frac12\|Au_h\|^\frac12 |z(\theta_t\omega)|^\frac12
\|Aw\|^\frac12\|A^\frac32w\|
& \nonumber\\
&\quad  \leq    \frac \nu {32} \|A^\frac32w\|^2  + C \|A^\frac12u_h\|  \|A u_h\|  |z(\theta_t\omega) |  \|Aw\|
\nonumber
\\
&\quad \leq     \frac \nu {32}\|A^\frac32w\|^2
+ C \|A^\frac12u_h\|^2  \|A u_h\|^2+    C| z(\theta_t\omega)|^2 \|Aw\|^2
\nonumber
\\
&\quad \leq     \frac \nu {32}\|A^\frac32w\|^2
+ C\left( \|A^\frac12u_h\|^2 +  |z(\theta_t\omega)|^2  \right) \|A w\|^2
+ C  \|A^\frac12u_h\|^2   \left(  1+ |z(\theta_t\omega) |^2  \right) ,
\ee
and
\be
  &C\|A^\frac12u_h\|^\frac12\|Au_h\|^\frac12 | z(\theta_t\omega)|  \|A^\frac32w\|
& \nonumber\\
&\quad \leq     \frac \nu {32}\|A^\frac32w\|^2  + C \|A^\frac12u_h\| \|A u_h\|   |z(\theta_t\omega)|^2
\nonumber
\\
&\quad \leq    \frac \nu {32}\|A^\frac32w\|^2
+ C \|A^\frac12u_h\|^2  \|A u_h\|^2   + C|z(\theta_t\omega) |^4
\nonumber
\\
&\quad \leq   \frac \nu {32}\|A^\frac32w\|^2
+ C  \|A^\frac12u_h\|^2     \left(  \|A w\|^2  +1+|z(\theta_t\omega) |^2  \right)
+ C|z(\theta_t\omega) |^4 ,
\ee
it follows that
\begin{align}
J_1(t)
&\leq \frac  \nu 8\|A^\frac32w\|^2
+C\left(\|A^\frac12u_h\|^2+\|A^\frac12w\|^2+|z(\theta_t\omega)|^2 \right)
\|Aw\|^2\nonumber\\
&\quad +C\left(1+\|A^\frac12u_h\|^4 \right)+C   |z(\theta_t\omega)|^4 .
\nonumber
\end{align}

For   the term $J_2(t)$, by  \ceqref{4.15} and the Gagliardo-Nirenberg inequality and the Young inequality we obtain
\begin{align}
J_2(t)
&\leq C\|A^{\frac12}u_h\|\|Aw\|^\frac12\|A^{\frac32}w\|^\frac 32 +C\|A^{\frac12}u_h\|\|Aw\|^\frac12|z(\theta_t\omega)|^\frac12\|A^\frac32w\|
 \nonumber\\
&\quad+C\|A^{\frac12}u_h\| |z(\theta_t\omega)|^\frac12\|A^{\frac32}w\|^\frac 32
+C\|A^{\frac12}u_h\| |z(\theta_t\omega)|  \|A^\frac32w\|\nonumber\\
&\leq \frac \nu 8\|A^\frac32w\|^2+C\left (1+\|A^\frac12u_h\|^4\right )\|Aw\|^2+C\left (1+\|A^\frac12u_h\|^8 \right)  +C |z(\theta_t\omega)|^4, \nonumber
\end{align}
and then  for $I_{5} (t)=J_1(t)+ J_2(t)$ we have the estimate
\begin{align}
I_{5}(t) &\leq  \frac \nu 4\|A^\frac32w\|^2
+C\left(1+ \|A^\frac12u_h\|^4+\|A^\frac12w\|^2+ |z(\theta_t\omega)|^2 \right)
\|Aw\|^2 \nonumber  \\
&\quad +C\left(1+\|A^\frac12u_h\|^8 \right)
+C  |z(\theta_t\omega)|^4. \label{sep6.2}
\end{align}

For the term $I_{6}(t)$, by the H\"{o}lder inequality and the Gagliardo-Nirenberg inequality and Poincar\'{e}'s inequality and the Young inequality, we have
\begin{align}\label{4.16}
I_{6}(t)&\leq\|A^{\frac12}(w+hz(\theta_t\omega))\|_{L^4}\|\nabla u\|_{L^4}\|A^\frac32w\|+\|w+hz(\theta_t\omega)\|_{L^\infty}\|Au\|_{L^2}\|A^\frac32w\|\nonumber\\
&\leq C\|A^{\frac12}(w+hz(\theta_t\omega))\|^\frac12\|A(w+hz(\theta_t\omega))\|^\frac12\|A^\frac12u\|^\frac12
\|Au\|^\frac12\|A^\frac32w\|\nonumber\\
& \quad +C\|A^\frac12(w+hz(\theta_t\omega))\|^\frac12\|A(w+hz(\theta_t\omega))\|^\frac12\|Au\|\|A^\frac32w\|\nonumber\\
&\leq C\|A^{\frac12}(w+hz(\theta_t\omega))\|^\frac12\|A(w+hz(\theta_t\omega))\|^\frac12\|A^\frac32w\|\nonumber\\
&\leq \frac \nu 8\|A^\frac32w\|^2+C\left(\|A^{\frac12}w\|+\|A^{\frac12}hz(\theta_t\omega)\| \right) \big (\|Aw\|
+\|Ahz(\theta_t\omega)\| \big )\nonumber\\
&\leq \frac \nu 8\|A^\frac32w\|^2+C\|Aw\|^2 +C |z(\theta_t\omega)|^2.
\end{align}
Substituting  \ceqref{4.18},  \ceqref{sep6.2} and  \ceqref{4.16} into  \ceqref{4.17}, by  \ceqref{mar5.3}  we deduce that
\begin{align}
 \frac{\d}{\d t}\|Aw\|^2+  \nu  \|A^{\frac32}w\|^2
&\leq C\left (1+\|A^\frac12u_h\|^4+\|A^\frac12w\|^2+|z(\theta_t\omega)|^2 \right )
\|Aw\|^2\nonumber\\
&\quad +C \left (1+\|A^\frac12u_h\|^8+|z(\theta_t\omega)|^4 \right) \nonumber \\
 &\leq C\left (1+\|A^\frac12 v\|^4 +|z(\theta_t\omega)|^4 \right )
\|Aw\|^2  \nonumber \\
&\quad +C \left (1+\|A^\frac12 v\|^8 +|z(\theta_t\omega)|^8 \right) .  \label{sep6.3}
\end{align}

 For $t >1$ and $s\in[t-1,t]$, applying Gronwall's lemma to  \ceqref{sep6.3}   it yields
\be
\|Aw(t)\|^2&\leq e^{C\int_s^t( 1+\|A^\frac12 v(\eta)\|^4 +|z(\theta_\eta\omega)|^4)\, \d \eta}\|Aw(s)\|^2 \nonumber\\
&\quad +\int_s^t Ce^{C\int_\tau^t  (1+\|A^\frac12 v\|^4+|z(\theta_\eta\omega)|^4)\, \d \eta}
 \Big (1+\|A^\frac12 v\|^8  +|z(\theta_\tau\omega)|^8\Big ) \, \d \tau\nonumber\\
&\leq Ce^{C\int_{t-1}^t(1+\|A^\frac12 v\|^8 +|z(\theta_\eta\omega)|^8)\, \d \eta} \nonumber \\
&\quad \times \left [\|Aw(s)\|^2+\int_{t-1}^t
\left(1+\|A^\frac12 v\|^8  +|z(\theta_\tau\omega)|^8\right)\d \tau \right]\\
& \leq Ce^{C\int_{t-1}^t(1+\|A^\frac12 v\|^8 +|z(\theta_\eta\omega)|^8)\, \d \eta} \Big(\|Aw(s)\|^2+ 1 \Big)
, \nonumber
\ee
where in the last inequality we have used the basic relation $x\leq e^x$ for $x\geq 0$.
Integrating w.r.t. $s$ over   $  [t-1,t] $ gives \begin{align}
\|Aw(t)\|^2
&\leq Ce^{C\int_{t-1}^t  (1+\|A^\frac12 v(\eta)\|^8+|z(\theta_\eta\omega)|^8)\, \d \eta}
 \left (\int_{t-1}^t\|Aw(s)\|^2\, \d s + 1\right). \nonumber
\end{align}
    Replacing $\omega$ with $\theta_{-t}\omega$ we obtain
\begin{align}\label{4.22}
\|Aw(t,\theta_{-t}\omega,w(0))\|^2
&\leq Ce^{C\int_{t-1}^t \|A^\frac12 v(\eta,\theta_{-t}\omega, v(0))\|^8 \d \eta + C\int^0_{-1} |z(\theta_{\eta}\omega) |^8
\, \d \eta}\nonumber\\
&\quad \times \left(\int_{t-1}^t\|Aw(s,\theta_{-t}\omega,w(0))\|^2\, \d s +1  \right) .
\end{align}
Note that    by  \ceqref{mar8.3}
it follows
\begin{align} \label{4.23}
  \int_{t-1}^t  \|A^\frac12 v(s,\theta_{-t}\omega,v(0))\|^8\, \d s \leq  |\zeta_3(\omega)|^4   ,  \quad t\geq T_{\mathfrak B}(\omega).
\end{align}
In addition, since $u$ is a trajectory within the global attractor $\mathcal A_0$ which is   bounded  in $H^2$,  by  \ceqref{mar8.3}  again we have
\begin{align}
\int_{t-1}^t\|Aw(s,\theta_{-t}\omega,w(0))\|^2\, \d s
&\leq2\int_{t-1}^t\|Av(s,\theta_{-t}\omega,v(0))\|^2\, \d s+2\int_{t-1}^t\|Au\|^2\, \d s\nonumber\\
&\leq  2\int_{t-1}^t \|Av(s,\theta_{-t}\omega,v(0))\|^2\, \d s  +  2\|\mathcal A \|_{H^2} \nonumber\\
&\leq \frac 2  \nu \, \zeta_3(\omega) +C ,  \quad t\geq T_{\mathfrak B}(\omega).  \label{mar8.7}
\end{align}
Therefore, substituting  \ceqref{4.23} and  \ceqref{mar8.7} into  \ceqref{4.22} we have
\begin{align}
\|Aw(t,\theta_{-t}\omega,w(0))\|^2
&\leq C e^{C|\zeta_3(\omega) |^4 + C\int^0_{-1} |z(\theta_{\eta}\omega) |^8
\, \d \eta} \left(   \zeta_3(\omega) +1
\right)  \nonumber \\
&=:\zeta_5(\omega),  \quad t\geq T_{\mathfrak B}(\omega).  \nonumber
\end{align}
The proof is complete.
\end{proof}

\subsection{$H^2$ random absorbing sets}
By  \cref{lem:H2bd} we are now able to construct an $H^2$ random absorbing set for the  random NS equations  \ceqref{2.2}. Nevertheless, for   later purpose we  make  the following  stronger estimate than  \cref{lem:H2bd}.
It will be crucial in the next section in deriving the local $(H,H^2)$-Lipschitz continuity of the equations.
\begin{lemma}\label{lemma4.4}
 Let  \cref{assum} hold and $f\in H$. Then
there is a tempered random variable $\rho(\cdot)$ such that  the solutions $v$ of the random NS equations  \ceqref{2.2} and the solutions $u$ of the deterministic equations  \ceqref{2} satisfy
\[
  \sup_{\varepsilon\in [0,1]} \|Av(t,\theta_{-t-\varepsilon}\omega,  v(0))- Au(t,u_0)\|^2     \leq  \rho(\omega) ,\quad t\geq T_{\mathfrak B} (\omega),
\]
whenever $v(0)\in \mathfrak B(\theta_{-t-\varepsilon} \omega)$ and $u(0)\in \mathcal A_0$, where $T_{\mathfrak B} (\cdot)$ is defined  in  \ceqref{timeB}.
\end{lemma}

\begin{proof}
Recall from  \ceqref{sep6.3} that
\begin{align*}
 \frac{\d}{\d t}\|Aw\|^2  &\leq C\left (1+\|A^\frac12 v\|^4 +|z(\theta_t\omega)|^4 \right )
\|Aw\|^2
+C \left (1+\|A^\frac12 v\|^8 +|z(\theta_t\omega)|^8 \right)\\
&\leq     C\left (1+\|A^\frac12 v\|^8 +|z(\theta_t\omega)|^8 \right )\left( \|Aw\|^2+1\right)   .
\end{align*}
For $t \geq 1$, integrating the  inequality from $(s,t)$ for $s\in (t-1,t)$  and then integrating over $s\in (t-1,t)$,  we obtain
\ben
 \|Aw(t)\|^2
& \leq    C\int_{t-1}^t   \left (1 +\|A^\frac12v(s)\|^8+|z(\theta_s\omega)|^8 \right ) \left( \|Aw(s)\|^2 +1\right)\d s.
\ee
For $\varepsilon\in [0,1]$, replacing $\omega$ with $\theta_{-t-\varepsilon}\omega$ yields
\ben
& \|Aw(t,\theta_{-t-\varepsilon}\omega, w(0))\|^2   \\
&\quad  \leq   C  \int_{t-1}^t   \left (1 +\|A^\frac12v(s,\theta_{-t-\varepsilon}\omega, v(0))\|^8+|z(\theta_{s-t-\varepsilon}\omega)|^8 \right ) \left( \|Aw(s)\|^2 +1\right) \d s   \\
&\quad  \leq   C  \int_{t+\varepsilon-2}^{t+\varepsilon}   \left (1 +\|A^\frac12v(s,\theta_{-t-\varepsilon}\omega, v(0))\|^8+|z(\theta_{s-t-\varepsilon}\omega)|^8 \right ) \left( \|Aw(s)\|^2 +1\right) \d s .
\ee
Note that from   \ceqref{mar8.3}  it follows
\ben
\sup_{\eta\in (t-2,t)}
 \|A^{\frac12}v (\eta ,\theta_{-t } \omega, v(0)) \|^8
&\leq  |\zeta_3(\omega)|^4 ,\quad t\geq T_{\mathfrak B} (\omega) .
\ee
Therefore,  by
\[
 \zeta_6(\omega):=   |\zeta_3(\omega)|^4 +\sup_{s\in(-2,0)} |z(\theta_s\omega)|^8+1
\]
we    defined a tempered random variable $\zeta_6$ such that
\ben
  \|Aw(t,\theta_{-t-\varepsilon}\omega, w(0))\|^2     \leq   C \zeta_6(\omega)  \int_{t+\varepsilon-2}^{t+\varepsilon}   \! \big( \|Aw(s,\theta_{-t-\varepsilon } \omega, w(0))\|^2 +1\big)\, \d s
\ee
for all $t\geq T_{\mathfrak B} (\omega)$
uniformly for $\varepsilon\in [0, 1]$.
Recall  from   \ceqref{mar8.7}  that
\[
 \int_{t+\varepsilon-2}^{t+\varepsilon}   \|Aw(s,\theta_{-t-\varepsilon } \omega, w(0))\|^2 \, \d s \leq \frac 1 \nu \, \zeta_3(\omega) +C
\]
for all $\varepsilon\in [0,1]$ and $t\geq  T_{\mathfrak B} (\omega) $, so we obtain
\be \label{rho}
\sup_{\varepsilon\in [0,1]}
  \|Aw(t,\theta_{-t-\varepsilon}\omega, w(0))\|^2     \leq   C \zeta_6(\omega)\left(\zeta_3(\omega) +1\right)=: \rho(\omega)
\ee
for all $ t\geq T_{\mathfrak B} (\omega)$
and then the lemma follows.
\end{proof}

Then by taking $\varepsilon =0$ we  now construct an $H^2$ random absorbing set.
\begin{theorem}[$H^2$ absorbing set]
\label{theorem4.1}
 Let  \cref{assum} hold and $f\in H$. Then the  RDS $\phi$ generated by the random  NS equations  \ceqref{2.2} has a  random absorbing set $\mathfrak{B}_{H^2}$, given as a random $H^2$ neighborhood of the global attractor $\mathcal{A}_0$ of the deterministic NS equations:
\begin{align}
\mathfrak{B}_{H^2}(\omega)=\left \{v\in H^2: \, {\rm dist}_{H^2}(v,\mathcal{A}_{0}) \leq \sqrt{\rho(\omega)} \, \right\}, \quad \omega\in\Omega, \nonumber
\end{align}
where $\rho(\cdot)$ is the tempered random variable given by  \ceqref{rho}.  As a consequence, the $\mathcal{D}_H$-random attractor $\mathcal{A}$ of  \ceqref{2.2} is a bounded and tempered random set in $H^2$.

\end{theorem}
\begin{proof}
This is a direct consequence of   \cref{lemma4.4}, see also   \cite[ Theorem 13]{cui}.\end{proof}

\section{The ($H,H^2$)-smoothing effect} \label{sec5}

In this section we derive the ($H,H^2$)-smoothing effect  of the random NS equations  \ceqref{2.2}. This effect  is essentially a
local $(H, H^2)$-Lipschitz continuity in initial values. It will be proved  step-by-step by   proving the local Lipschitz continuity in $H$, the  $(H,H^1)$-smoothing and finally the  $(H,H^2)$-smoothing. For ease of notations we shall denote by
 $$ \bar v(t,\omega, \bar v(0)):= v_1(t,\omega, v_1(0))-v_2(t,\omega, v_2(0)) $$
the difference between two solutions.

\subsection{Lipschitz continuity in $H$}

\begin{lemma}
 Let  \cref{assum} hold and $f\in H$.
 Then for  any tempered set $\mathfrak D\in \D_H $ there exist random variables $t_{\mathfrak D}(\cdot)$ and $L_1(\mathfrak D,\cdot) $ such that any
 two solutions $v_1$ and $v_2$ of NS equation  \ceqref{2.2} corresponding to initial values  $v_{1,0},$ $v_{2,0}$ in $ \mathfrak D( \theta_{-t_{\mathfrak D}(\omega)}\omega),$   respectively,  satisfy
\ben
 & \left \| v_1 \! \left ( t_{\mathfrak D} (\omega) ,\theta_{-t_{\mathfrak D}(\omega)}\omega,  v_{1,0}\right)
 - v_2 \! \left ( t_{\mathfrak D} (\omega) ,\theta_{-t_{\mathfrak D}(\omega)}\omega, v_{2,0}\right) \right \|^2  \\[0.8ex]
 &\quad
 \leq L_1({\mathfrak D}, \omega) \|{v}_{1,0}-v_{2,0}\|^2  ,\quad \omega \in \Omega.
\ee

\end{lemma}
\begin{proof}
The  difference $\bar{v}  =v_1-v_2$ of   solutions of  \ceqref{2.2}  satisfies
\begin{align}\label{5.1}
\frac{\d \bar{v}}{
\d t}+  \nu  A\bar{v}+B( v_1+hz(\theta_t\omega))-B( v_2+hz(\theta_t\omega))=0.
\end{align}
Multiplying the  equation   \ceqref{5.1} by $\bar{v}$ and integrating over $\mathbb T^2$, by integration by parts  we have
\ben
\frac{1}{2}\frac{\d}{\d t}\|\bar{v}\|^2+  \nu \|A^{\frac12}\bar{v}\|^2&=- \big (B(\bar{v},v_1+hz(\theta_t\omega)),\bar{v} \big )
-\big (B(v_2+hz(\theta_t\omega),\bar{v}),\bar{v}\big )\nonumber\\
&=-\big (B(\bar{v},v_1+hz(\theta_t\omega)),\bar{v} \big )\nonumber\\
&\leq \|\bar{v}\|^2_{L^4}\|\nabla(v_1+hz(\theta_t\omega))\|\nonumber\\
&\leq C\|\bar{v}\|\|A^{\frac12}\bar{v}\| \left (\|A^\frac12v_1\|+\|A^\frac12hz(\theta_t\omega)\| \right)\nonumber\\
&\leq \frac\nu2\|A^{\frac12}\bar{v}\|^2+C \left(\|A^\frac12v_1\|^2+|z(\theta_t\omega)|^2 \right )\|\bar{v}\|^2.
\ee
Therefore,
\[
\frac{\d}{\d t}\|\bar{v}\|^2+  \nu \|A^{\frac12}\bar{v}\|^2
\leq C\left (\|A^\frac12v_1\|^2+|z(\theta_t\omega)|^2 \right) \|\bar{v}\|^2.
\]
Applying Gronwall's lemma, we deduce that
\begin{align} \label{mar9.1}
 & \|\bar{v}(t)\|^2+  \nu \int_0^te^{\int_s^tC(\|A^\frac12v_1(\tau) \|^2+|z(\theta_\tau\omega)|^2)\, \d \tau}\|A^{\frac12}\bar{v} (s) \|^2\, \d s \nonumber \\
& \quad
\leq e^{\int_0^tC(\|A^\frac12v_1(\tau)\|^2+|z(\theta_\tau\omega)|^2)\, \d \tau}\|\bar{v}(0)\|^2, \quad t>0.
\end{align}

Note that applying  \ceqref{4.29} we   have
\ben
   \int_0^t \|A^\frac12v_1(\tau,\theta_{-t}\omega, v_{1,0})\|^2 \, \d \tau &\leq e^{ \left(\frac 8 \alpha -3  \right)\lambda t}  \int_0^t e^{ \left( \frac 8 \alpha -3  \right) \lambda (\tau-t)} \|A^\frac12v_1 (\tau,\theta_{-t}\omega, v_{1,0})\|^2  \,  \d \tau  \\
 &   \leq \frac {2}{\alpha\nu} e^{ \left(\frac 8 \alpha -3  \right)\lambda t}  \left( e^{-\lambda t} \|v_{1,0}\|^2 +\zeta_1(\omega) \right) ,\quad t\geq T_1(\omega),
 \ee
so for  \ceqref{mar9.1} we have
\ben
 \|\bar v(t,\theta_{-t}\omega, \bar v(0))\|^2
 & \leq  e^{\int_0^tC \left (\|A^\frac12v_1(\tau,\theta_{-t}\omega, v_{1,0})\|^2+|z(\theta_{\tau-t} \omega)|^2 \right)   \d \tau }
 \|\bar{v}(0)\|^2 \\
 &\leq e^{C  e^{ \left(\frac 8 \alpha -3  \right)\lambda t}  \left( e^{-\lambda t} \|v_{1,0}\|^2 +\zeta_1(\omega) \right) +C\int_{-t}^0|z(\theta_{\tau} \omega)|^2 \, \d \tau } \|\bar{v}(0)\|^2
\ee
for all $t\geq T_1 (\omega)$.
Since $v_{1,0} $ belongs to $\mathfrak D(\theta_{-t}\omega) $ which is tempered, there exists a random variable $ t_{\mathfrak D} (\omega)  \geq T_1(\omega)$  such that
\[
 e^{-\lambda t_{\mathfrak D}(\omega) } \|v_{1,0}\|^2 \leq  e^{-\lambda t_{\mathfrak D}(\omega) } \left \|  \, \mathfrak D \!  \left (\theta_{-t_{\mathfrak D}(\omega)} \omega \right ) \right\|^2 \leq 1,\quad \omega\in \Omega.
\]
In addition, since $\mathfrak D$ is pullback absorbed by the absorbing set $\mathfrak B$,  this $t_{\mathfrak D}(\omega)$ is chosen large enough such that
\be \label{mar19.2}
 \phi \left ( t_{\mathfrak D}(\omega) , \theta_{-t_{\mathfrak D}(\omega)}\omega, \mathfrak D(\theta_{- t_{\mathfrak D}(\omega) }\omega) \right ) \subset \mathfrak B(\omega), \quad \omega\in \Omega .
\ee
Hence, we define a random variable by
\[
L_1({\mathfrak D}, \omega)=e^{C  e^{ \left(\frac 8 \alpha -3  \right)\lambda  t_{\mathfrak D}(\omega)}    ( 1+\zeta_1(\omega)  ) + C\int_{-t_{\mathfrak D}(\omega)}^0|z(\theta_{\tau} \omega)|^2 \, \d \tau } ,\quad \omega\in \Omega,
\]
then
\be \label{mar18.8}
\left \|\bar v  \! \left ( t_{\mathfrak D} (\omega) ,\theta_{-t_{\mathfrak D}(\omega)}\omega, \bar v(0) \right) \right \|^2
 \leq L_1({\mathfrak D}, \omega) \|\bar{v}(0)\|^2  ,\quad
\ee
whenever $v_{1,0},$ $v_{2,0}\in \mathfrak D \big ( \theta_{-t_{\mathfrak D}( \omega)}\omega \big),$  $\omega\in \Omega$.
\end{proof}

\subsection{$(H,H^1)$-smoothing}

The $(H,H^1)$-smoothing will be proved  by two steps. Since  we have constructed  by    \ceqref{4.19} an $H^1$ random absorbing set $\mathfrak B$, we begin with initial values lying  in the absorbing set $\mathfrak B$, and then consider the  initial values  in every tempered set $\mathfrak D $ in $\D_H$.

\begin{lemma}[$(H,H^1)$-smoothing on $\mathfrak B$]
 Let  \cref{assum} hold and $f\in H$. Then for   the random absorbing set $\mathfrak B$ defined by  \ceqref{4.19}     there exist random variables $\tau_\omega$ and $L_2(\mathfrak B,\omega) $ such that any
 two solutions $v_1$ and $v_2$ of the random NS  equation  \ceqref{2.2} corresponding to initial values  $v_{1,0},$ $v_{2,0}$ in $ \mathfrak B( \theta_{- \tau_\omega}\omega),$   respectively,  satisfy
\be  \label{mar18.9}
 \left \| v_1  (  \tau_\omega,\theta_{-\tau_\omega}\omega,  v_{1,0} )
 - v_2   ( \tau_\omega,\theta_{-\tau_\omega}\omega, v_{2,0} ) \right \|^2_{H^1}
 \leq L_2({\mathfrak B}, \omega) \|{v}_{1,0}-v_{2,0}\|^2  ,\quad \omega \in \Omega.
\ee
\end{lemma}

\begin{proof}
Multiplying   \ceqref{5.1} by $A\bar{v}$ and integrating over $\mathbb T^2$, by integration by parts we have
\begin{align}\label{5.2}
\frac{1}{2}\frac{\d}{\d t}\|A^\frac12\bar{v}\|^2+ \nu \|A\bar{v}\|^2&=- \big (B(\bar{v},v_1+hz(\theta_t\omega)),A\bar{v}\big )
-\big (B(v_2+hz(\theta_t\omega),\bar{v}),A\bar{v}\big ) .
\end{align}
Applying the H\"{o}lder inequality, the Gagliardo-Nirenberg inequality and the Young inequality, we deduce that
\begin{align}\label{5.3}
 \big| \big (B(\bar{v},v_1+hz(\theta_t\omega)),A\bar{v}\big )\big|
 &\leq \|\bar{v}\|_{L^\infty}\|\nabla(v_1+hz(\theta_t\omega))\|\|A\bar{v}\|\nonumber\\
&\leq C\|\bar{v}\|^\frac12\|A\bar{v}\|^\frac32
\left (\|A^\frac12v_1\|+\|A^\frac12hz(\theta_t\omega)\| \right)\nonumber\\
&\leq C\|A^\frac12\bar{v}\|^\frac12\|A\bar{v}\|^\frac32 \left (\|A^\frac12v_1\|+ |z(\theta_t\omega) | \right)\nonumber\\
&\leq \frac \nu 4\|A\bar{v}\|^2+C \left (\|A^\frac12v_1\|^4+|z(\theta_t\omega)|^4\right)\|A^\frac12\bar{v}\|^2,
\end{align}
and that
\begin{align}\label{5.4}
 \big| \big (B(v_2+hz(\theta_t\omega),\bar{v}),A\bar{v}\big )\big| &\leq \|A^\frac12\bar{v}\|_{L^4}\|v_2+hz(\theta_t\omega)\|_{L^4}\|A\bar{v}\|\nonumber\\
&\leq C\|A^\frac12\bar{v}\|^\frac12\|A\bar{v}\|^\frac32 \left(\|A^\frac12v_2\|+\|A^\frac12hz(\theta_t\omega)\| \right)\nonumber\\
&\leq \frac \nu 4\|A\bar{v}\|^2+C\left (\|A^\frac12v_2\|^4+|z(\theta_t\omega)|^4\right)\|A^\frac12\bar{v}\|^2.
\end{align}
Substituting  \ceqref{5.3}- \ceqref{5.4} into  \ceqref{5.2}, we have
\begin{align} \label{mar9.2}
\frac{\d}{\d t}\|A^\frac12\bar{v}\|^2+  \nu \|A\bar{v}\|^2
\leq C\left(\|A^\frac12v_1\|^4+\|A^\frac12v_2\|^4+|z(\theta_t\omega)|^4\right)\|A^\frac12\bar{v}\|^2.
\end{align}
Applying Gronwall's lemma, we obtain for $s\in(t-1, t-\frac 12)$, $t\geq 1$, that
\begin{align}
 & \|A^\frac12\bar{v}(t)\|^2
 +  \nu  \int_s^te^{\int_\eta^tC(\|A^\frac12v_1(\tau)\|^4+\|A^\frac12v_2(\tau)\|^4+|z(\theta_\tau\omega)|^4)\, \d \tau}\|A\bar{v}(\eta) \|^2 \, \d \eta
\nonumber\\
&\quad \leq e^{ \int_s^tC(\|A^\frac12v_1(\tau) \|^4+\|A^\frac12v_2(\tau)\|^4+|z(\theta_\tau\omega)|^4)\, \d \tau}\|A^\frac12\bar{v}(s)\|^2, \nonumber
\end{align}
and then integrating w.r.t. $s$ over $(t-1,t-\frac 12)$  yields
\ben
 & \|A^\frac12\bar{v}(t)\|^2
 + \nu \int_{t-\frac 12} ^te^{\int_\eta^tC(\|A^\frac12v_1(\tau)\|^4+\|A^\frac12v_2(\tau)\|^4+|z(\theta_\tau\omega)|^4)\, \d \tau}\|A\bar{v}(\eta) \|^2 \, \d \eta
\nonumber\\
&\quad \leq  2 \int_{t-1} ^{t-\frac 12}   e^{ \int_s^tC(\|A^\frac12v_1(\tau) \|^4+\|A^\frac12v_2(\tau)\|^4+|z(\theta_\tau\omega)|^4)\, \d \tau}\|A^\frac12\bar{v}(s)\|^2  \, \d s \\
& \quad \leq  2e^{ \int_{t-1}^tC(\|A^\frac12v_1(\tau) \|^4+\|A^\frac12v_2(\tau)\|^4+|z(\theta_\tau\omega)|^4)\, \d \tau}  \int_{t-1} ^{t-\frac 12} \|A^\frac12\bar{v}(s)\|^2 \, \d s
.
\ee
By  \ceqref{mar9.1}  this gives
\begin{align}
 & \|A^\frac12\bar{v}(t)\|^2
 + \nu\int_{t-\frac 12} ^te^{\int_\eta^tC(\|A^\frac12v_1(\tau)\|^4+\|A^\frac12v_2(\tau)\|^4+|z(\theta_\tau\omega)|^4)\, \d \tau}\|A\bar{v}(\eta) \|^2 \, \d \eta \nonumber  \\
  &\quad  \leq  2e^{ \int_{0}^tC(\|A^\frac12v_1(\tau) \|^4+\|A^\frac12v_2(\tau)\|^4+|z(\theta_\tau\omega)|^4+1)\, \d \tau}
   \|\bar{v}(0)\|^2
,\quad t\geq 1. \label{mar19.1}
\end{align}

Since the absorbing set $\mathfrak B$ itself belongs to the attraction universe $\D_H$,  it pullback absorbs itself. Hence  there is a random variable $\tau_\omega \geq 1$ such that
\be \label{mar18.2}
 \phi \big(\tau_\omega,\theta_{-\tau_\omega}\omega, \B(\theta_{-\tau_\omega}\omega)\big)
 \subset \B(\omega), \quad \omega\in \Omega.
\ee
 Replacing $\omega$ with $\theta_{-\tau_\omega} \omega$ in  \ceqref{mar19.1}  we have  the estimate at $t=\tau_\omega$ that
\begin{align}
  \big \|
 A^{\frac 12}  \bar {v}(\tau_\omega,\theta_{-\tau_\omega}\omega,  \bar v(0))  \big \| {^2}
 \leq  2e^{ C\int_{0}^{\tau_\omega}  \sum_{i=1}^2\|A^\frac12v_i(s, \theta_{-\tau_\omega}\omega, v_i(0)) \|^4 \d s+ C\int^0_{-\tau_\omega} (|z(\theta_s\omega)|^4+1) \d s}
   \|\bar v(0) \|^2 .  \label{mar18.1}
\end{align}

 In order to bound the right-hand side of  \ceqref{mar18.1} we recall from  \ceqref{4.27} that
\ben
\frac{\d}{\d t}\|A^{\frac12}v\|^2
&\leq    C\left ( 1 +  |z(\theta_t\omega)|^2 \right)  \|A^{\frac12}v\|^2  +
C \left(1 +|z(\theta_t\omega)|^6  \right) .
\ee
Since the absorbing set $\mathfrak B $ is bounded and tempered in $H^1$, for  initial values $v(0)$ in   $\mathfrak B(\omega)$ we apply  Gronwall's lemma to obtain
\ben
 \|A^{\frac12}v(s,\omega, v(0))\|^2
&\leq e^{ \int_0^s C(1+|z(\theta_\tau\omega)|^2)\, \d \tau}\|A^{\frac12}v( 0)\|^2
\\
&\quad +C\int_0^s e^{ \int_\eta^s C(1+ |z (\theta_\tau\omega)|^2)\, \d \tau} \left (1 +|z(\theta_\eta\omega)|^6 \right) \d \eta \\
&\leq e^{ \int_0^s C(1+|z(\theta_\tau\omega)|^2)\, \d \tau}
\left[ \zeta_2(\omega)+C\int_0^s   \left (1 +|z(\theta_\eta\omega)|^6 \right) \d \eta \right]
\ee
for $s\geq 0$
uniformly for $v(0)\in  \mathfrak B(\omega)$,   where $\zeta_2(\omega) $ is the square of the radius of $\mathfrak B(\omega)$ in $H^1$. As a consequence, particularly for $s\in (0, \tau_\omega)$ we have
\ben
  \sup_{s\in (0, \tau_\omega)} \|A^{\frac12}v(s,\omega, v(0))\|^2
&\leq e^{ \int_0^{\tau_\omega}  C(1+|z(\theta_\tau\omega)|^2)\, \d \tau}
\left[\zeta_2(\omega)+C\int_0^{\tau_\omega}   \left (1 +|z(\theta_\eta\omega)|^6 \right) \d \eta \right] ,
\ee
and then, replacing $\omega $ with $ \theta_{-\tau_\omega} \omega$,
\ben
&  \sup_{s\in (0, \tau_\omega)} \|A^{\frac12}v(s, \theta_{-\tau_\omega} \omega, v(0))\|^2 \\
&\quad \leq e^{ \int^0_{-\tau_\omega}  C(1+|z(\theta_\tau\omega)|^2)\, \d \tau}
\left[ \zeta_2( \theta_{-\tau_\omega}\omega)+C\int^0_{-\tau_\omega}   \left (1 +|z(\theta_\eta\omega)|^6 \right) \d \eta \right] =: \zeta_7(\omega) .
\ee
 Hence,
\[
    \int_{0}^{\tau_\omega}  \|A^\frac12v(s, \theta_{-\tau_\omega}\omega, v(0)) \|^4 \, \d s  \leq  \tau_\omega |\zeta_7(\omega)|^2
\]
uniformly for  $v(0) $ in $  \mathfrak B( \theta_{-\tau_\omega} \omega)$ and thus  for  \ceqref{mar18.1} we have
 \ben
  \|
 A^{\frac 12} \bar {v}(\tau_\omega,\theta_{-\tau_\omega}\omega, \bar v(0))  \|^2
&  \leq  2e^{ C \tau_\omega  |\zeta_7( \omega)|^2 + C\int^0_{-\tau_\omega} (|z(\theta_s\omega)|^4+1) \d s}
   \| \bar v(0)  \|^2  \\
   &\leq 2e^{ C \tau_\omega  |\zeta_7( \omega)|^2 }
   \| \bar v(0)  \|^2.
\ee
Therefore, with the random variable
\[
  L_2(\mathfrak B, \omega):=  2e^{ C \tau_\omega  |\zeta_7( \omega)|^2 }
,\quad \omega\in \Omega,
\]
being defined we have proved the lemma.
\end{proof}

Now we derive  the   $(H,H^1)$-smoothing for initial values from any tempered set $\mathfrak D$ in $H$.
\begin{lemma} \label{lem:H1}
 Let  \cref{assum} hold and $f\in H$. Then for   any tempered set   $\mathfrak D\in \D_H$ there exist random variables $ T_1(\mathfrak D ,\cdot )  $ and $L_3(\mathfrak D,\cdot) $ such that any
 two solutions $v_1$ and $v_2$ of system  \ceqref{2.2} corresponding to initial values  $v_{1,0},$ $v_{2,0}$ in $ \mathfrak D( \theta_{-T_1(\mathfrak D,\omega)}\omega),$   respectively,  satisfy
\be \label{mar19.4}
  \left \| v_1  (T_1 ,\theta_{-T_1}\omega,  v_{1,0} )
 - v_2   ( T_1 ,\theta_{-T_1}\omega, v_{2,0} ) \right \|^2_{H^1}
 \leq L_3({\mathfrak D}, \omega) \|{v}_{1,0}-v_{2,0}\|^2  ,
\ee
where $T_1= T_1(\mathfrak D, \omega)$, $\omega\in \Omega$. In addition, $ v (T_1 ,\theta_{-T_1}\omega,  v_{0}) \in \mathfrak B(\omega)$ for all $v_0\in \mathfrak D(\theta_{-T_1}\omega) $.
\end{lemma}

\begin{proof}
Since ${\mathfrak D} \in \D_H$   is pullback absorbed by $\mathfrak B$, recall from  \ceqref{mar19.2} that  we have a random variable $  t_{\mathfrak D} (\cdot) $ such that
\be \label{mar19.5}
 \phi \left (  t_{\mathfrak D}(\omega) , \theta_{-  t_{\mathfrak D}(\omega)}\omega, \mathfrak D(\theta_{-  t_{\mathfrak D}(\omega)}\omega) \right) \subset \mathfrak B(\omega), \quad \omega\in \Omega.
\ee
Therefore,    for all $ \omega\in \Omega $   we have
\ben
&   \left \|  \bar v  \!\left( \tau_\omega+   t_{\mathfrak D}(\theta_{-\tau_\omega} \omega)   , \, \theta_{-\tau_\omega  - t_{\mathfrak D}(\theta_{-\tau_\omega} \omega)}
 \omega, \, \bar v(0)\right)
   \right \|^2_{H^1}
    \nonumber \\
         &\quad = \left \| \bar v  \! \left( \tau_\omega  , \, \theta_{-\tau_\omega   }
 \omega, \,  \bar v \big(   t_{\mathfrak D}(\theta_{-\tau_\omega} \omega) ,\,  \theta_{-\tau_\omega- t_{\mathfrak D}(\theta_{-\tau_\omega} \omega)  }
 \omega,  \, \bar v(0)\big) \right)
   \right \|^2_{H^1} \nonumber
   \\
 &
  \quad \leq  L_2  (\mathfrak B,   \omega  )   \left \|  \bar v \!  \left( t_{\mathfrak D}(\theta_{-\tau_\omega} \omega),\,  \theta_{-\tau_\omega- t_{\mathfrak D}(\theta_{-\tau_\omega} \omega)    }
 \omega, \,  \bar v(0) \right)
   \right \|^2  \quad \text{(by  \ceqref{mar18.9})} \nonumber \\
 &\quad \leq     L_2  (\mathfrak B,   \omega  )   L_1  \big (\mathfrak D, \theta_{-\tau_\omega}\omega  \big)  \|\bar{v}(0)\|^2  \quad \text{(by  \ceqref{mar18.8})}
 \nonumber
\ee
for all $v_{1,0},$ $v_{2,0}\in \mathfrak D \left (  \theta_{-\tau_\omega  - t_{\mathfrak D}(\theta_{-\tau_\omega} \omega)} \omega \right)$.  With
 \ben
 & T_1({\mathfrak D,\omega}) := \tau_\omega + t_{\mathfrak D}(\theta_{-\tau_\omega} \omega) ,  \\
 & L_3(\mathfrak D,\omega):=  L_2 \left (\mathfrak B,    \omega \right )   L_1 \left(\mathfrak D, \theta_{-\tau_\omega}\omega  \right) ,
 \ee
we have proved  \ceqref{mar19.4}.

From the definition  \ceqref{mar19.5}  of $ t_{\mathfrak D}(\theta_{-\tau_\omega} \omega)$ it is clear that, for all $v_0\in \mathfrak D(\theta_{-T_1}\omega) $,
 $$
 y:= v \left(    t_{\mathfrak D}(\theta_{-\tau_\omega} \omega)   , \, \theta_{- t_{\mathfrak D}(\theta_{-\tau_\omega} \omega)} \circ \theta_{-\tau_\omega}
 \omega, \,  v(0)\right) \in \mathfrak B(\theta_{-\tau_\omega}\omega).
$$
In addition, by the definition  \ceqref{mar18.2}  of $\tau_\omega$ we know that
\[
 v(\tau_\omega, \theta_{-\tau_\omega} \omega, y) \in \mathfrak B(\omega).
\]
Therefore,
\ben
    v\left (T_1 ,\theta_{-T_1}\omega,  v_{0}\right)
    & =  v \left( \tau_\omega+   t_{\mathfrak D}(\theta_{-\tau_\omega} \omega)   , \, \theta_{- t_{\mathfrak D}(\theta_{-\tau_\omega} \omega)}(\theta_{-\tau_\omega}
 \omega), \,  v(0)\right)\\
 &= v(\tau_\omega, \theta_{-\tau_\omega} \omega, y)
  \in \mathfrak B(\omega)
\ee
 for all $v_0\in \mathfrak D(\theta_{-T_1}\omega) $, as desired.
 \end{proof}

\subsection{$(H , H^2)$-smoothing}

We next prove the $(H , H^2)$-smoothing property. For ease of analysis we  restrict again ourselves on the absorbing set $\mathfrak B $ first. We  begin with the following useful estimate.

\begin{lemma}  Let  \cref{assum} hold and $f\in H$.
 Then the solutions $v$ of  \ceqref{2.2} corresponding to initial values  $v(0)$  satisfy the estimate
\be
\sup_{s\in [t-1,t]}
 \|Av(s,\theta_{-t}\omega,  v(0)) \|^2
  \leq     \rho(\omega) +\|\mathcal A_0\|_{H^2}^2,
 \quad  t\geq T_{\mathfrak B} (\omega)+ 1,\label{mar9.4}
\ee
whenever  $v(0) \in\mathfrak B (\theta_{-t} \omega)$, where  $T_{\mathfrak B} $ is the random variable given in    \ceqref{timeB}.
\end{lemma}
\begin{proof}
For  $ t\geq T_{\mathfrak B} (\omega)+1$ and $s\in (t-1,t)$ so that $s\geq T_{\mathfrak B} (\omega)$, by
  \cref{lemma4.4} we have
\[
 \|Av(s,\theta_{-t}\omega,  v(0)) \|^2    =
 \|Av(s,\theta_{-s-(t-s)}\omega,  v(0)) \|^2
 \leq     \rho(\omega) +\|\mathcal A_0\|_{H^2}^2
\]
uniformly for $s\in [t-1,t]$ and $v(0) \in\mathfrak B (\theta_{-t} \omega)$.
\end{proof}

\begin{lemma}[$(H^1, H^2)$-smoothing on $\mathfrak B$]\label{lemma5.1}  Let  \cref{assum} hold and $f\in H$. Then there  exist  random variables $T_\omega$ and $ L_4(\mathfrak B, \omega )$ such that   any two solutions $v_1$ and $v_2$ of system  \ceqref{2.2} corresponding to initial values  $v_{1,0},$ $v_{2,0}$ in $\mathfrak{B}(\theta_{-T_\omega}\omega)$, respectively,   satisfy
\[
 \big \|v_1(T_\omega,\theta_{-T_\omega}\omega,v_{1,0})-v_2(T_\omega,\theta_{-T_\omega}\omega,v_{2,0}) \big \|{^2_{H^2}}\leq  L_4(\mathfrak B, \omega)\|v_{1,0}-v_{2,0}\|^2_{H^1}, \quad \omega\in \Omega.
\]
\end{lemma}

\begin{proof}
Multiplying the  equation   \ceqref{5.1} by $A^2\bar{v}$ and integrating over $\mathbb T^2$, by integration by parts we obtain
\begin{align}\label{5.5}
\frac{1}{2}\frac{\d}{\d t}\|A\bar{v}\|^2+ \nu  \|A^\frac32\bar{v}\|^2&=- \big (B(\bar{v},v_1+hz(\theta_t\omega)),\, A^2\bar{v}  \big )
- \big (B(v_2+hz(\theta_t\omega),\bar{v}),\, A^2\bar{v} \big ).
\end{align}
Applying the H\"{o}lder inequality, the Gagliardo-Nirenberg inequality and the Young inequality, we deduce that
\begin{align}\label{5.6}
 \big| \big(B(\bar{v},v_1+hz(\theta_t\omega)),\, A^2\bar{v}\big)
 \big| &\leq \|A^\frac12\bar{v}\|_{L^\infty}\|\nabla(v_1+hz(\theta_t\omega))\|\|A^\frac32\bar{v}\|  \nonumber\\
&\quad
 +\|\bar{v}\|_{L^\infty}\|A(v_1+hz(\theta_t\omega))\|\|A^\frac32\bar{v}\|
 \nonumber\\
&\leq C\|A\bar{v}\|^\frac12\|A^\frac32\bar{v}\|^\frac32
\left (\|A^\frac12v_1\|+\|A^\frac12h\||z(\theta_t\omega)| \right )\nonumber\\
&\quad +C\|\bar{v}\|^\frac12\|A\bar{v}\|^\frac12\|A^\frac32\bar{v}\| \big(\|Av_1\|+\|Ah\||z(\theta_t\omega)| \big)\nonumber\\
&\leq \frac  \nu 2\|A^\frac32\bar{v}\|^2+C\|A\bar{v}\|^2 \left (\|A^\frac12v_1\|^4+|z(\theta_t\omega)|^4 \right)\nonumber\\
&\quad +C\|A^\frac12\bar{v}\|^2 \left (\|Av_1\|^4+|z(\theta_t\omega)|^4 \right),
\end{align}
and that
\begin{align}\label{5.7}
\big| \big(B(v_2+hz(\theta_t\omega),\bar{v}),\, A^2\bar{v}\big)\big|
&\leq \|A^\frac12(v_2+hz(\theta_t\omega))\|\|\nabla\bar{v}\|_{L^\infty}\|A^\frac32\bar{v}\| \nonumber\\
&\quad
 +\|v_2+hz(\theta_t\omega)\|_{L^4}\|A\bar{v}\|_{L^4}\|A^\frac32\bar{v}\|\nonumber\\
&\leq C\|A\bar{v}\|^\frac12\|A^\frac32\bar{v}\|^\frac32
\left (\|A^\frac12v_2\|+\|A^\frac12h\||z(\theta_t\omega)| \right)\nonumber\\
&\quad +C \|A^\frac12(v_2+hz(\theta_t\omega)
)\|\|A\bar{v}\|^\frac12\|A^\frac32\bar{v}\|^\frac32 \nonumber\\
&\leq \frac  \nu 2\|A^\frac32\bar{v}\|^2+C\|A\bar{v}\|^2 \left (\|A^\frac12v_2\|^4+|z(\theta_t\omega)|^4 \right).
\end{align}

Substituting  \ceqref{5.6} and  \ceqref{5.7} into  \ceqref{5.5} yields
\begin{align}
\frac{\d}{\d t}\|A\bar{v}\|^2
&\leq C\|A\bar{v}\|^2 \left ( \sum_{j=1}^2 \|A^\frac12v_j\|^4 +|z(\theta_t\omega)|^4 \right )   +C\|A^\frac12\bar{v}\|^2 \left (\|Av_1\|^4+|z(\theta_t\omega)|^4 \right). \nonumber
\end{align}
For $s\in(t-1,t)$ and $t\geq 1$, applying Gronwall's lemma  we deduce that
\begin{align}
&\|A\bar{v}(t)\|^2
-e^{\int_s^tC( \sum_{i=1}^2 \|A^\frac12v_i\|^4 +|z(\theta_\tau\omega)|^4)\, \d \tau}\|A\bar{v}(s)\|^2\nonumber\\
&\quad \leq \int_s^tCe^{\int_\eta^tC( \sum_{i=1}^2 \|A^\frac12v_i\|^4+|z(\theta_\tau\omega)|^4)\, \d \tau}
\|A^\frac12\bar{v}\|^2 \left(\|Av_1\|^4+|z(\theta_\eta\omega)|^4\right) \d \eta\nonumber\\
&\quad \leq Ce^{\int_{t-1}^tC( \sum_{i=1}^2 \|A^\frac12v_i\|^4+|z(\theta_\tau\omega)|^4)\, \d \tau}
\int_{t-1}^t\|A^\frac12\bar{v}\|^2\left(\|Av_1\|^4+|z(\theta_\eta\omega)|^4 \right) \d \eta. \nonumber
\end{align}
Integrating w.r.t. $s$ on $(t-1,t)$ yields
\begin{align}\label{5.9}
&\|A\bar{v}(t)\|^2
-\int_{t-1}^te^{\int_s^tC( \sum_{i=1}^2 \|A^\frac12v_i\|^4+|z(\theta_\tau\omega)|^4)\, \d \tau}\|A\bar{v}(s)\|^2\, \d s\nonumber\\
&\quad \leq Ce^{\int_{t-1}^tC( \sum_{i=1}^2 \|A^\frac12v_i\|^4+|z(\theta_\tau\omega)|^4)\, \d \tau}
\int_{t-1}^t\|A^\frac12\bar{v}\|^2\left (\|Av_1\|^4+|z(\theta_\eta\omega)|^4\right) \d \eta.
\end{align}

Note that since our initial values $v_{1,0}$ and $v_{2,0}$   belong to the $H^1$ random  absorbing set  $\mathfrak B$,  applying Gronwall's lemma to  \ceqref{mar9.2} it follows   that
\begin{align}
 & \|A^\frac12\bar{v}(t)\|^2 + \nu \int_0^te^{\int_s^tC( \sum_{i=1}^2 \|A^\frac12v_i\|^4+|z(\theta_\tau\omega)|^4)\, \d \tau}\|A\bar{v}(s) \|^2\, \d s
\nonumber\\
&\quad \leq e^{\int_0^tC( \sum_{i=1}^2 \|A^\frac12v_i\|^4+|z(\theta_\tau\omega)|^4)\, \d \tau}\|A^\frac12\bar{v}(0)\|^2,\quad t\geq 1. \nonumber
\end{align}
As a consequence, we obtain the   inequalities
\begin{align}
& \nu \int_{t-1}^te^{\int_s^tC( \sum_{i=1}^2 \|A^\frac12v_i\|^4+|z(\theta_\tau\omega)|^4)\, \d \tau}\|A\bar{v}(s)\|^2\, \d s\nonumber\\
&\quad \leq e^{\int_{0}^tC( \sum_{i=1}^2 \|A^\frac12v_i\|^4+|z(\theta_\tau \omega)|^4)\, \d \tau}\|A^\frac12\bar{v}(0)\|^2, \nonumber
\end{align}
and
\begin{align}
&
\int_{t-1}^t\|A^\frac12\bar{v}\|^2 \left (\|Av_1\|^4+|z(\theta_\eta\omega)|^4\right) \d \eta\nonumber\\
&\quad \leq e^{\int_0^tC( \sum_{i=1}^2 \|A^\frac12v_i\|^4+|z(\theta_\tau\omega)|^4)\, \d \tau}\|A^\frac12\bar{v}(0)\|^2
\int_{t-1}^t \!  \left (\|Av_1\|^4+|z(\theta_\eta\omega)|^4\right) \d \eta.\nonumber
\end{align}
Substituting these inequalities  into  \ceqref{5.9}, we have
\begin{align}
 \|A\bar{v}(t)\|^2
&\leq Ce^{\int_0^tC( \sum_{i=1}^2 \|A^\frac12v_i\|^4+|z(\theta_\tau\omega)|^4)\, \d \tau}\|A^\frac12\bar{v}(0)\|^2
\int_{t-1}^t  \! \left (1+\|Av_1\|^4+|z(\theta_\eta\omega)|^4\right)  \d \eta   \nonumber
\end{align}
for $  t\geq 1$.
Replacing $\omega$ by $\theta_{-t}\omega$, we obtain   for $  t\geq 1$ that
\begin{align} \label{mar11.1}
 \|A\bar{v}(t,\theta_{-t}\omega,\bar{v}(0))\|^2
&\leq Ce^{\int_0^tC( \sum_{i=1}^2\|A^\frac12v_i(\tau,\theta_{-t}\omega,v_i(0))\|^4
+|z(\theta_{\tau-t}\omega)|^4)\, \d \tau}\|A^\frac12\bar{v}(0)\|^2\nonumber\\
&\quad  \times \int_{t-1}^t  \! \left ( \|Av_1(\eta,\theta_{-t}\omega,v_{1,0})\|^4+|z(\theta_{\eta-t}\omega)|^4+1\right) \d \eta.
\end{align}

Recall from  \ceqref{mar9.4} that
  \[
\sup_{s\in [t-1,t]}
 \|Av(s,\theta_{-t}\omega,  v(0)) \|^2
  \leq     \rho(\omega) +\|\mathcal A_0\|_{H^2}^2,
 \quad  t\geq T_{\mathfrak B} (\omega)+1.
\]
Hence,  taking
\[
T_\omega :=T_{\mathfrak B} (\omega)+1, \quad \omega\in \Omega,
\]
we have
\begin{align}
 & \int_{T_\omega -1}^{T_\omega }  \left ( \|Av_1(\eta,\theta_{-T_\omega }\omega,v_{1,0})\|^4+|z(\theta_{\eta-T_\omega }\omega)|^4\right)  \d \eta  \nonumber\\
 &\quad
 \leq  \left( \rho(\omega) +\|\mathcal A_0\|_{H^2}^2 \right)^2 + \int_{-1}^0 |z(\theta_{\eta}\omega)|^4 \, \d \eta =: \zeta_8(\omega) . \label{mar11.2}
 \end{align}

On the other hand, recall  from  \ceqref{4.27} that
$$
\frac{\d}{\d t}\|A^{\frac12}v\|^2 \leq C \left ( 1+ |z(\theta_t\omega)|^2 \right )\|A^{\frac12}v\|^2 + C \left (1 +|z(\theta_t\omega)|^6 \right) .
$$
Gronwall's lemma gives
\ben
 \|A^{\frac 12} v(t,\omega, v(0))\|^2
  & \leq  e^{ \int^t_0 C(1+|z(\theta_\tau\omega)|^2) \, \d \tau} \|A^{\frac 12} v(0)\|^2  \\
 &\quad
 +\int_0^t Ce^{ \int^t_s C(1+|z(\theta_\tau\omega)|^2) \, \d \tau} \left (1+ |z(\theta_s\omega)|^6 \right)  \d s \\
 &\leq  e^{ \int^t_0 C(1+|z(\theta_\tau\omega)|^2) \, \d \tau} \left[ \|A^{\frac 12} v(0)\|^2
 +\int_0^t  C \left(1+ |z(\theta_s\omega)|^6 \right)  \d s \right]  .
\ee
Since
\[
\int_0^t  C \left (1+ |z(\theta_s\omega)|^6 \right)  \d s \leq e^{ \int^t_0 C(1+|z(\theta_\tau\omega)|^6) \, \d \tau},
\]
the estimate is simplified to
$$
 \|A^{\frac 12} v(t,\omega, v(0))\|^2
   \leq  \left( \|A^{\frac 12} v(0)\|^2 +1 \right) e^{ \int^t_0 C(1+|z(\theta_\tau\omega)|^6) \, \d \tau} ,
   \quad t>0.
$$
Therefore, uniformly for   $v(0) \in \mathfrak B(\theta_{-t}\omega)$ we have
 \begin{align}
 \int_0^t
 \|A^{\frac 12} v(\eta, \theta_{-t} \omega, v(0))\|^2 \, \d \eta
 &   \leq  \left( \| \mathfrak B  (\theta_{-t}\omega)\|^2_{H^1} +1 \right) \int_0^t e^{ \int^\eta_0 C(1+|z(\theta_{\tau-t} \omega)|^6) \, \d \tau}  \, \d \eta \nonumber \\
 &   \leq  \big( \zeta_2(\theta_{-t}\omega)  +C \big)  e^{ Ct+C \int^0_{-t}  |z(\theta_{\tau} \omega)|^6 \, \d \tau}   ,
   \quad t>0  , \nonumber
\end{align}
and, particularly for $t=T_\omega$,
 \begin{align}  \label{mar11.3}
 & \int_0^{T_\omega}
 \|A^{\frac 12} v(\eta, \theta_{-{T_\omega}} \omega, v(0))\|^2 \, \d \eta \nonumber \\
 &   \quad  \leq  \big( \zeta_2(\theta_{-{T_\omega}}\omega)  +C \big)  e^{ C {T_\omega}+C \int^0_{-{T_\omega}}  |z(\theta_{\tau} \omega)|^6 \, \d \tau} =:\zeta_9(\omega).
\end{align}

Summarizing, substituting      \ceqref{mar11.2}  and  \ceqref{mar11.3}  into  \ceqref{mar11.1}  we obtain at $t=T_\omega$ that
\begin{align}
 \|A\bar{v}(T_\omega ,\theta_{-T_\omega }\omega,\bar{v}(0))\|^2
&\leq Ce^{ C \zeta_9(\omega) } \zeta_8(\omega) \|A^\frac12\bar{v}(0)\|^2. \nonumber
\end{align}
Therefore, the lemma is proved with
\[
 L_4 (\mathfrak B, \omega) :=   Ce^{ C \zeta_9(\omega) } \zeta_8(\omega) ,\quad \omega\in \Omega. \qedhere
\]
\end{proof}

Now we are ready to conclude the following main result of this section.
\begin{theorem}[$(H, H^2)$-smoothing] \label{theorem5.6}
 Let  \cref{assum} hold and $f\in H$. Then for   any tempered set $\mathfrak D \in \D_H$ there
  exist  random variables $T_{{\mathfrak D}} (\cdot)  $   and $ L_{\mathfrak D}(\cdot )$ such that  any two solutions $v_1$ and $v_2$ of system  \ceqref{2.2} corresponding to initial values   $v_{1,0},$ $ v_{2,0}$ in $\mathfrak D \left (\theta_{-T_{\mathfrak D}(\omega)}\omega \right)$, respectively, satisfy
  \begin{align}
  &
  \big\|v_1 \!  \big (T_{\mathfrak D}(\omega),\theta_{-T_{\mathfrak D}(\omega)}\omega,v_{1,0}\big)
  -v_2 \big (T_{\mathfrak D}(\omega),\theta_{-T_{\mathfrak D}(\omega)}\omega,v_{2,0} \big) \big \|{^2_{H^2}}  \notag \\[0.8ex]
&\quad \leq  L_{\mathfrak D} (\omega)\|v_{1,0}-v_{2,0}\|^2,\quad \omega\in \Omega.\label{sep5.1}
\end{align}
\end{theorem}
\begin{proof}
By  \cref{lem:H1}, for the tempered set   $\mathfrak D\in \D_H$ there exist random variables $ T_1(\mathfrak D ,\cdot )  $ and $L_3(\mathfrak D,\cdot) $ such that
\be   \label{mar19.6}
  \left \| \bar v\left (T_1 ,\theta_{-T_1}\omega,  \bar v(0)\right)
  \right \|^2_{H^1}
 \leq L_3({\mathfrak D}, \omega) \|\bar {v} (0)\|^2  ,
\ee
where $T_1= T_1(\mathfrak D, \omega)$. In addition, $ v\left (T_1 ,\theta_{-T_1}\omega,  v_{0}\right) \in \mathfrak B(\omega)$ for all $v_0\in \mathfrak D(\theta_{-T_1}\omega) $, $\omega\in \Omega$.
 Therefore, by  \cref{lemma5.1} and  \ceqref{mar19.6},
\ben
  \left \| \bar v \big( T_\omega +T_1, \theta_{-T_\omega-T_1}
 \omega, \bar v(0) \big)
   \right \|^2_{H^2}
 &  =  \left \| \bar v \big( T_\omega  , \theta_{-T_\omega }
 \omega, \bar v  ( T_1, \theta_{-T_\omega-T_1}
 \omega, \bar v(0)  )\big)
   \right \|^2_{H^2}  \\
   &\leq  L_4(\mathfrak B, \omega) \left \|  \bar v \big  ( T_1, \theta_{-T_\omega-T_1}
 \omega, \bar v(0)  \big)
   \right \|^2_{H^1}  \  \text{(by  \cref{lemma5.1})}\\
   &\leq L_4(\mathfrak B,\omega)  L_3 \big (\mathfrak D, \theta_{-T_\omega}\omega \big) \|\bar v(0)\|^2  \  \text{(by  \ceqref{mar19.6})} ,
\ee
where $T_1=T_1(\theta_{-T_\omega}  \omega)$,
uniformly for $v_{1,0},$ $ v_{2,0}\in \B(\theta_{-T_\omega-T_1(\theta_{-T_\omega} \omega)}\omega)$, $  \omega\in \Omega$. Hence,
with the random variables
\ben
  & T_{\mathfrak D}( \omega) := T_\omega +T_1(\theta_{-T_\omega} \omega) , \\
  &L_{\mathfrak D}(\omega) :=  L_4(\mathfrak B,\omega)   L_3 \big (\mathfrak D, \theta_{-T_\omega}\omega \big),
\ee
being defined we have the theorem.
  \end{proof}

\section{The $(H,H^2)$-random attractor}\label{sec6}
With the above preparations being made we are now able to  show that the  random attractor  $\mathcal A$ of the random NS equation   \ceqref{2.2} is in fact a finite-dimensional  $(H,H^2)$-random attractor. More precisely,  \cref{theorem4.6} is now strengthened to the following.

\begin{theorem} \label{theorem5.1}
 Let  \cref{assum} hold and $f\in H$. Then the  RDS $\phi$ generated by the random NS equation   \ceqref{2.2} has  a tempered $(H,H^2)$-random attractor $\mathcal A $. In addition, $\mathcal A$ has finite fractal dimension in $H^2$: there exists   a constant $ d>0$ such that
\[
d_f^{H^2} \! \big ( \A(\omega) \big)  \leq d, \quad \omega \in \Omega.
\]
\end{theorem}
\begin{proof}
 In  \cref{theorem4.1}  we have constructed an $H^2$ random absorbing set $\mathfrak B_{H^2}$, which is tempered and closed in $H^2$,  while the  $(H,H^2)$-smoothing property  \ceqref{sep5.1}   implies the $(H,H^2)$-asymptotic compactness of $\phi$. Therefore,  by  \cref{lem:cui18} we conclude that $\A$ is indeed an $(H,H^2)$-random attractor of $\phi$.
 Since the fractal dimension in $H$ of  $\A$ is finite  (\cref{theorem4.6}), the   finite-dimensionality  in $H^2$ follows from   the $(H,H^2)$-smoothing property  \ceqref{sep5.1}  again, in view of  \cref{lem:cui}.
\end{proof}

 \begin{remark}
 The $(H,H^2)$-smoothing property  \ceqref{sep5.1}  and the finite fractal dimension in $H^2$ of the global attractor    are new even for     deterministic NS equations.
 \end{remark}

\section*{Acknowledgements}

  H. Liu was supported by the National Natural Science Foundation of China (Nos. 12271293, 11901342) and the project of Youth Innovation Team of Universities of Shandong Province (No. 2023KJ204).  H. Cui was   supported  by NSFC  11801195 and Ministerio de Ciencia e  Innovaci\'on  (Spain) and FEDER (European Community) under grant PID2021-122991NB-C21. J. Xin was supported  by the Natural Science Foundation of Shandong Province (No. ZR2023MA002).

\end{document}